\documentclass{amsart}
\usepackage{amsfonts}

\newtheorem{theorem}{Theorem}
\theoremstyle{plain}

\newtheorem{condition}{Condition}

\newtheorem{corollary}{Corollary}

\newtheorem{definition}{Definition}
\newtheorem{example}{Example}

\newtheorem{lemma}{Lemma}

\newtheorem{proposition}{Proposition}
\newtheorem{remark}{Remark}

\numberwithin{equation}{section}
\input{tcilatex}

\begin{document}
\title[Trajectory of concentration]{Newton's law for a trajectory of
concentration of solutions to nonlinear Schrodinger equation}
\author{Anatoli Babin }
\address{Department of Mathematics \\
University of California at Irvine \\
Irvine, CA 92697-3875, U.S.A.}
\email{ababine@math.uci.edu}
\thanks{Supported by AFOSR grant number FA9550-11-1-0163.}
\author{Alexander Figotin}
\address{Department of Mathematics \\
University of California at Irvine \\
Irvine, CA 92697-3875, U.S.A.}
\email{afigotin@math.uci.edu}
\date{}
\dedicatory{Dedicated to the memory of M.I.Vishik }

\begin{abstract}
One of important problems in mathematical physics concerns derivation of
point dynamics from field equations. The most common approach to this
problem is based on WKB\ method. Here we describe a different method based
on the concept of trajectory of concentration. When we applied this method
to nonlinear Klein-Gordon equation, we derived relativistic Newton's law and
Einstein's formula for inertial mass. Here we apply the same approach to
nonlinear Schrodinger equation and derive non-relativistic Newton's law for
the trajectory of concentration.
\end{abstract}

\maketitle

\section{Introduction}

A. Einstein remarks in his letter to Ernst Cassirer of March 16, 1937, \cite[%
pp. 393-394]{Stachel EBZ}: "One must always bear in mind that up to now we
know absolutely nothing about the laws of motion of material points from the
standpoint of "classical field theory." For the mastery of this problem,
however, no special physical hypothesis is needed, but "only" the solution
of certain mathematical problems". In this paper we treat this circle of
problems. Namely, we derive here point dynamics governed by Newton's law
from wave dynamics using concepts and methods of the classical field theory.
That continues our study of a neoclassical model of electric charges which
is designed to provide an accurate description of charge interaction with
the electromagnetic field from macroscopic down to atomic scales, \cite{BF5}-%
\cite{BF9}.

In this paper we derive Newtonian dynamics for localized solutions of
Nonlinear Schrodinger equations (NLS) in the asymptotic limit when the \
solutions converge to delta-functions. The derivation is based on a method
of concentrating solutions which was developed in our previous paper \cite%
{BF9} where we derived the relativistic dynamics and Einstein's formula for
localized solutions of the Nonlinear Klein-Gordon (KG) equation.

In many problems of physics and mathematics a common way to establish a
relation between wave and point dynamics is by means of the WKB method, see
for instance \cite{MaslovFedorjuk}, \cite[Sec. 7.1]{Nayfeh}.\ We remind that
the WKB\ method is based on the quasiclassical ansatz for solutions to a
hyperbolic partial differential equation and their asymptotic expansion. The
leading term of the expansion satisfies the eikonal equation, and
wavepackets and their energy propagate along the so-called characteristics
of this equations. Consequently, these characteristics represent the point
dynamics and are determined from the corresponding system of ODEs which can
be interpreted as a law of motion or a law of propagation. The construction
of the characteristics involves only local data. Asymptotic derivation of
the Newtonian dynamics which does not rely on the fast oscillation of
solutions as the WKB\ method does was developed for soliton-like solutions
of the Nonlinear Schrodinger and Nonlinear Klein-Gordon equations in a
number of papers, see \cite{BronskiJer00}, \cite{FrTY}, \cite{JonssonFGS}, 
\cite{LongS08}. In the mentioned papers the derivation of the limit
Newtonian dynamics of the soliton center relies on the asymptotic analysis\
of an ansatz structure of soliton-like solutions. \emph{The proposed here
approach also relates waves governed by certain PDE's in asymptotic regimes
to the point dynamics but it differs fundamentally from the mentioned methods%
}. In particular, our approach is not based on any specific ansatz and it is
not not entirely local but rather it is semi-local. \ Our semi-local method
is in the spirit of the Ehrenfest theorem and is based on an analysis of a
sequence of concentrating solutions which are not subjected to structural
conditions. Compared to the mentioned approaches which use the WKB method or
soliton-like structure of solutions, our arguments only use very mild
assumptions of a localization of solutions in a small neighborhood of the
trajectory combined with the systematic use there of integral identities
derived from the energy-momentum and density conservation laws for the NLS.
Our method allows to derive Newton's law under minimal restrictions on the
time and spatial dependence of the variable coefficients and on the
nonlinearity and allows to consider arbitrarily long time intervals.

The nonlinear Schrodinger (NLS)\ equation we study in this paper is of the
form 
\begin{equation}
\mathrm{i}\tilde{\partial}_{t}\psi =\frac{\chi }{2m}\left[ -\tilde{\nabla}%
^{2}\psi +G_{a}^{\prime }\left( \psi ^{\ast }\psi \right) \psi \right] .
\label{NLS}
\end{equation}%
In the above equation $\psi =\psi \left( t,\mathbf{x}\right) $\ is a complex
valued wave function, $G_{a}^{\prime }$ is a real valued nonlinearity, and
the \textit{covariant differention operators} $\tilde{\partial}_{t}$ and $%
\tilde{\nabla}$ are defined by%
\begin{equation}
\tilde{\partial}_{t}=\partial _{t}+\frac{\mathrm{i}q}{\chi }\varphi ,\quad 
\tilde{\nabla}=\nabla -\frac{\mathrm{i}q}{\chi \mathrm{c}}\mathbf{A},
\label{covNLS}
\end{equation}%
where $\varphi \left( t,\mathbf{x}\right) $, $\mathbf{A}\left( t,\mathbf{x}%
\right) $ are given twice differentiable functions of time $t$ and spatial
variables $\mathbf{x}\in \mathbb{R}^{3}$ \ interpreted as potentials of
external electric and magnetic fields. The quantity $q\left\vert \psi
\right\vert ^{2}$ is naturally interpreted as the charge density with $q$
being the value of the charge.

In the case when the potentials $\varphi $ and $\mathbf{A}$ are zero, the
charge can be considered as "free" and the NLS\ equation (\ref{NLS})\ has
localized solutions corresponding to resting and uniformly moving charges.
Newton's law involves acceleration, and to produce an accelerated motion
non-zero potentials are required. One may expect point-like dynamics for
strongly localized solutions, and the nonlinearity $G^{\prime }$ provides
existence of localized solutions of the NLS equation. The only condition
imposed on the nonlinearity (in addition to natural continuity assumptions,
see Condition \ref{Dregnon}) is the existence of a positive radial solution $%
\psi _{1}=\mathring{\psi}_{1}\left( \left\vert \mathbf{x}\right\vert \right)
\ $ of the steady-state equation\ for a free charge: 
\begin{equation}
-\nabla ^{2}\mathring{\psi}_{1}+G_{1}^{\prime }(\left\vert \mathring{\psi}%
_{1}\right\vert ^{2})\mathring{\psi}_{1}=0.  \label{free}
\end{equation}%
We assume that $\mathring{\psi}\left( \left\vert \mathbf{x}\right\vert
\right) $\ is twice continuously differentiable for all $\mathbf{x}\in 
\mathbb{R}^{3}$ and is square integrable, that is%
\begin{equation}
\int_{\mathbb{R}^{3}}\left\vert \mathring{\psi}\right\vert ^{2}\mathrm{d}%
\mathbf{x}=\upsilon _{0}<\infty .  \label{norm}
\end{equation}%
Importantly, we assume that the nonlinearity depends explicitly on the size
parameter $a>0$ as follows:\ 
\begin{equation}
G_{a}^{\prime }\left( s\right) =a^{-2}G_{1}^{\prime }\left( a^{3}s\right)
,\quad s>0.  \label{Gas}
\end{equation}%
Note that the steady-state equation with a general $\ a>0$ \ \ takes the
form 
\begin{equation}
\nabla ^{2}\mathring{\psi}_{a}+G_{a}^{\prime }(\left\vert \mathring{\psi}%
_{a}\right\vert ^{2})\mathring{\psi}_{a}=0  \label{nop40}
\end{equation}%
and has the following solution: 
\begin{equation}
\mathring{\psi}_{a}\left( r\right) =a^{-3/2}\mathring{\psi}_{1}\left(
a^{-1}r\right) ,\quad r=\left\vert x\right\vert \geq 0,  \label{nrac5}
\end{equation}%
where $\mathring{\psi}_{1}$ is the solution to equation (\ref{free}) where $%
a=1$. Obviously $\mathring{\psi}_{a}$ satisfies the normalization condition (%
\ref{norm}) \ with the same $\upsilon _{0}$. According to (\ref{nrac5}) $a$
can be interpreted as a natural size parameter which describes $\mathring{%
\psi}_{a}$.

A typical example of the ground state $\mathring{\psi}_{a}$ is the Gaussian $%
\mathring{\psi}_{a}=a^{-3/2}e^{-r^{2}/a^{2}}$ corresponding to a logarithmic
nonlinearity $G_{a}^{\prime }$ discussed in Example \ref{Egau}. Evidently, $%
\upsilon _{0}^{-1}\mathring{\psi}_{a}^{2}$ converges to Dirac's
delta-function $\delta \left( \mathbf{x}\right) $ as $a\rightarrow 0$. \ We
study the behavior of localized solutions of the NLS equation in the
asymptotic regime $a\rightarrow 0$.

The crucial role in our analysis is played by the key concept of
concentrating solutions. Very roughly speaking, \ solutions $\psi =\psi _{n}$
concentrate at a given trajectory $\mathbf{\hat{r}}\left( t\right) $ if
their charge densities $q\left\vert \psi \right\vert ^{2}\left( \mathbf{x}%
,t\right) $ restricted to $R_{n}$-neighborhoods $\Omega _{n}$ \ of $\mathbf{%
\hat{r}}\left( t\right) $ locally converge to $q_{\infty }\left( t\right)
\delta \left( \mathbf{x}-\mathbf{\hat{r}}\left( t\right) \right) $ as 
\begin{equation}
a_{n}\rightarrow 0,\qquad R_{n}\rightarrow 0,\qquad a_{n}/R_{n}\rightarrow
0,\qquad n\rightarrow \infty .  \label{arar}
\end{equation}%
Now we describe the concept of \emph{concentrating} solutions in a little
more detail. There are two relevant spatial scales for the NLS: the
microscopic size parameter $a$ which determines the size of a free charge
and the macroscopic length scale $R_{\mathrm{ex}}$ \ of order 1 at which the
potentials $\varphi $ and $\mathbf{A}$ vary significantly. We introduce the
third intermediate spatial scale $R_{n}$ which can be called the \emph{%
confinement scale} and we assume that $\psi _{n}$ asymptotically vanish at
the boundary $\partial \Omega _{n}=\left\{ \left\vert \mathbf{x}-\mathbf{%
\hat{r}}\left( t\right) \right\vert =R_{n}\right\} $ of the neighborhood $%
\Omega _{n}$. \ According to (\ref{arar}) $R_{n}/a\rightarrow \infty $, \ \
therefore \ this assumption is quite natural for solutions with typical
spatial scale $a$ \ which are localized at $\mathbf{\hat{r}}$. \ At the same
time, $R_{n}/R_{\mathrm{ex}}\rightarrow 0$, \ \ therefore general potentials 
$\varphi \left( t,\mathbf{x}\right) $, $\mathbf{A}\left( t,\mathbf{x}\right) 
$ \ can be replaced successfully in $\Omega _{n}$\ by their linearizations
at $\mathbf{x}=\mathbf{\hat{r}}\left( t\right) $. The exact definition of
the concentrating solutions (or concentrating asymptotic solutions) involves
two major assumptions:

\begin{enumerate}
\item[(1)] volume integrals over the balls $\Omega _{n}$ of the charge
density $q\left\vert \psi \right\vert ^{2}$ and the momentum density $%
\mathbf{P}$ should be bounded;

\item[(2)] surface integrals over the spheres $\partial \Omega _{n}$ of
certain quadratic expressions which involve $\psi $ and its first order
derivatives and originate from the elements of the energy-momentum tensor of
the NLS equation tend to zero as $R_{n}\rightarrow 0$, $a_{n}/R_{n}%
\rightarrow 0$. Details of these assumptions can be found in Definitions \ref%
{Dlocconverges} and \ref{Dconcas}.
\end{enumerate}

A concise formulation of our main result, Theorem \ref{Cq0sa}, is as
follows. We prove that \emph{if a sequence of asymptotic solutions of the
NLS\ equation (\ref{NLS}) concentrates at a trajectory }$\mathbf{\hat{r}}%
\left( t\right) $\emph{\ then this trajectory must satisfy the Newton
equation} 
\begin{equation}
m\partial _{t}^{2}\mathbf{\hat{r}}=\mathbf{f}_{\mathrm{Lor}}\left( t,\mathbf{%
\hat{r}}\right) ,  \label{Newt0}
\end{equation}%
where $\mathbf{f}_{\mathrm{Lor}}$ is the Lorentz force which is defined by
the classical formula 
\begin{equation}
\mathbf{f}_{\mathrm{Lor}}\left( t,\mathbf{\hat{r}}\right) =q\mathbf{E}\left(
t,\mathbf{\hat{r}}\right) +\frac{q}{\mathrm{c}}\partial _{t}\mathbf{\hat{r}}%
\times \mathbf{B}\left( t,\mathbf{\hat{r}}\right) ,  \label{fLor0}
\end{equation}%
with the electric and magnetic fields $\mathbf{E}$ and $\mathbf{B}$ defined
in terms of the potentials $\ \varphi ,\mathbf{A}$ in (\ref{covNLS})
according to the standard formula%
\begin{equation}
\mathbf{E}=-\nabla \varphi -\frac{1}{\mathrm{c}}\partial _{t}\mathbf{A}%
,\qquad \mathbf{B}=\nabla \times \mathbf{A.}  \label{maxw3a}
\end{equation}%
Notice that that according to Newton's equation (\ref{Newt0}) the parameter $%
m$ entering the NLS equation (\ref{NLS}) can be interpreted as the mass of
the charge.

Now we would like to comment on certain subjects related to our main Theorem %
\ref{Cq0sa}. First of all, the assumptions imposed on the concentrating
solutions are not restrictive. Indeed, the solutions are local and have only
to be defined in a tubular neighborhood of the trajectory and no initial or
boundary conditions are imposed on them. Second of all, only the charge and
the momentum conservation laws are used in the derivation of the theorem.
Since only conservation laws are used for the argument, the asymptotic
solutions for the NLS equation\ are introduced as functions $\psi $\ for
which the charge and the momentum conservation laws hold approximately.
Namely, certain integrals which involve quantities that enter the
conservation laws vanish in an asymptotic limit, see Definition \ref{Dconcas}
for details. Remarkably, such mild restrictions still completely determine
all possible trajectories of concentration, that is the trajectory $\mathbf{%
\hat{r}}\left( t\right) $ is uniquely defined by the initial data $\mathbf{%
\hat{r}}\left( t_{0}\right) $ and $\partial _{t}\mathbf{\hat{r}}\left(
t_{0}\right) $ as a solution of (\ref{Newt0}).

Equation (\ref{Newt0}) yields a \emph{necessary condition} for solutions of
the NLS\ equation to concentrate at a trajectory. To obtain this necessary
condition we have to derive the point dynamics governed by an ordinary
differential equation (\ref{Newt0})\ from the dynamics of waves governed by
a partial differential equation. The concept of "\emph{concentration of
functions at a given trajectory} $\mathbf{\hat{r}}\left( t\right) $", see
Definition \ref{Dlocconverges}, is the first step in relating spatially
localized fields $\psi $ to point trajectories. The definition of
concentration of functions has a sufficient flexibility to allow for general
regular trajectories $\mathbf{\hat{r}}\left( t\right) \ $and plenty of
functions are localized about the trajectory. But if a sequence of functions
concentrating at a given trajectory\emph{\ also satisfy or asymptotically
satisfy the conservation laws for the NLS equation then, according to
Theorems \emph{\ref{Cq0s} and \ref{Cq0sa}}, the trajectory and the limit
energy must satisfy Newton's equation}.

To derive Newton's equation (\ref{Newt0}), we consider $\psi \left( \mathbf{x%
},t\right) $ restricted to a narrow tubular neighborhood of the trajectory
of radius $R_{n}$ and consider adjacent charge centers 
\begin{equation}
\mathbf{r}_{n}\left( t\right) =\frac{1}{\bar{\rho}_{n}}\int_{\left\vert 
\mathbf{x-\hat{r}}\right\vert \leq R_{n}}\mathbf{x}q\left\vert \psi
_{n}\right\vert ^{2}\mathrm{d}\mathbf{x},\qquad \bar{\rho}%
_{n}=\int_{\left\vert \mathbf{x-\hat{r}}\right\vert \leq R_{n}}\,q\left\vert
\psi _{n}\right\vert ^{2}\mathrm{d}\mathbf{x},
\end{equation}%
of the concentrating solutions. Then we infer integral equations for the
restricted momentum and for the adjacent centers from the momentum
conservation law and the continuity equation, and pass to the limit as $%
R_{n}\rightarrow 0$, $a_{n}\rightarrow 0$ as is shown in the proofs of
Theorems \ref{Tconctr}-\ref{Cq0s}.\ Note that the determination of the
restricted charge density $\bar{\rho}$ and a similar momentum density
involves integration over a large relative to $a$ spatial domain of radius $%
R_{n}$, $R_{n}>>a$. Therefore it is natural to call our method of
determination of the point trajectories \emph{semi-local}. This semi-local
feature applied to the nonlinear Klein-Gordon equation in \cite{BF8}, \cite%
{BF9} allowed us to derive Einstein's relation between mass and energy with
the energy being an integral quantity.

If the form factor $\mathring{\psi}_{1}\left( \theta \right) $ decays fast
enough as $\theta \rightarrow \infty $ (faster than $\theta ^{-2}$), we also
prove the converse statement (Theorem \ref{Ttrconv}) to Theorem \ref{Cq0sa}.
Namely, \emph{if }$\hat{r}\left( t\right) $\emph{\ is a solution of equation
(\ref{Newt0}) then it is a trajectory of concentration for a sequence of
asymptotic solutions to the NLS. }The proof of this theorem is based on an
explicit construction of \emph{wave-corpuscle solutions} of the NLS of the
form $\psi \left( t,\mathbf{x}\right) =\mathrm{e}^{\mathrm{i}S}\mathring{\psi%
}\left( \mathbf{x-\hat{r}}\right) $ for a general enough class of potentials 
$\varphi $ and $\mathbf{A}$ with a certain phase function $S$. We choose
auxiliary potentials $\varphi _{\mathrm{aux}}$ and $\mathbf{A}_{\mathrm{aux}%
} $ from this class to approximate $\varphi $ and $\mathbf{A}$ at $\mathbf{%
\hat{r}}\left( t\right) $. The wave-corpuscle solutions constructed for the
auxiliary potentials $\varphi _{\mathrm{aux}}$ and $\mathbf{A}_{\mathrm{aux}%
} $ \ turn out to form a sequence of concentrating asymptotic solutions of
the NLS. Therefore, for given\ twice continuously differentiable potentials $%
\varphi $ and $\mathbf{A}$, a trajectory is a trajectory of concentration of
asymptotic solutions of the NLS if and only if this trajectory satisfies
Newton's law (\ref{Newt0}). The phase function $S$ of the wave-corpuscle can
be interpreted as the phase of de Broglie wave (see \cite{BF5} for details).
Hence, the described relation between Newtonian trajectories and
concentrating asymptotic solutions can be interpreted as what is known as
the wave-particle duality.

The paper's structure \ is as follows. In the following Subsection \ref%
{Sconclaw} we define the charge density, current density and momentum
density for a solution of NLS and write down corresponding conservation
laws. The customary field theory derivation of the conservation laws is
given in Appendix 1. In Section \ref{SNLSpoint} we describe the class of
trajectories and in Subsection \ref{SlocNLS} \ we formulate restrictions on
the fields $\varphi $ and $\mathbf{A}$ and give the definition of solutions
concentrating to a trajectory (Definition \ref{Dlocconverges}). In
Subsections \ref{Sprop} and \ref{Sder} \ we prove our main result on the
characterization of trajectories of concentration of solutions to the NLS
equation (Theorem \ref{Cq0s}).

In Section \ref{swcpres} we consider "wave-corpuscles" that exactly preserve
their shape in an accelerated motion and describe a class of potentials $%
\varphi $,$\mathbf{A}$ \ which allow for such a motion. In Subsection \ref%
{Swcconc} we prove that the wave corpuscles provide an example of
concentrating solutions to the NLS equation (Theorem \ref{Twcconc}). In
Section \ref{Sconcas} we provide an extension of results of Section \ref%
{SNLSpoint} to asymptotic solutions of the NLS equation, including the main
Theorem \ref{Cq0sa}. In Section \ref{Spoint} \ we prove that if a trajectory
satisfies Newton's law then it is a trajectory of concentration of
asymptotic solutions of the NLS equation.

\subsection{Conservation laws for NLS\label{Sconclaw}}

The \emph{charge density} $\rho $ and \emph{current density} $\mathbf{J}$
are defined by the formulas 
\begin{gather}
\rho =q\psi \psi ^{\ast },  \label{schr8} \\
\mathbf{J}=-\mathrm{i}\frac{q\chi }{2m}\left( \psi ^{\ast }\tilde{\nabla}%
\psi -\psi \tilde{\nabla}^{\ast }\psi ^{\ast }\right) ,  \label{schr8a}
\end{gather}%
and for solutions of (\ref{NLS}) they satisfy the \emph{continuity equation}%
: 
\begin{equation}
\partial _{t}\rho +\nabla \cdot \mathbf{J}=0.  \label{schr9}
\end{equation}%
One can verify the validity of the continuity equation (\ref{schr9})
directly by simply multiplying (\ref{NLS}) by $\psi ^{\ast }$\ and taking
the imaginary part. An alternative derivation based on the Lagrangian field
formalism is provided in Section \ref{sschrodinger}. The \emph{momentum
density} is defined by the formula%
\begin{equation}
\mathbf{P}=\mathrm{i}\frac{\chi }{2}\left[ \psi \cdot \tilde{\nabla}^{\ast
}\psi ^{\ast }-\psi ^{\ast }\cdot \tilde{\nabla}\psi \right] ,  \label{schr5}
\end{equation}%
and it is proportional to the current density, namely%
\begin{equation}
\mathbf{P}=\frac{m}{q}\mathbf{J.}  \label{PmqJ}
\end{equation}%
The momentum density satisfies the \emph{momentum equation} \ of the form 
\begin{equation}
\partial _{t}\mathbf{P}+\partial _{i}T^{ij}=\mathbf{f},  \label{momNLS}
\end{equation}%
where $\mathbf{f}$ is the Lorentz force density defined by the formula 
\begin{equation}
\mathbf{f}=\rho \mathbf{E}+\frac{1}{\mathrm{c}}\mathbf{J}\times \mathbf{B,}
\label{flors}
\end{equation}%
$T^{ij}$ are the entries of the energy-momentum tensor defined by (\ref%
{emfr10}), (\ref{Tij}), and the EM fields $\mathbf{E}$, $\mathbf{B}$ are
defined in terms of the potentials $\ \varphi $, $\mathbf{A}$ by the
standard formulas (\ref{maxw3a}). The momentum equation can be derived from
the NLS equation using multiplication by $\tilde{\nabla}^{\ast }\psi ^{\ast
} $ \ and some rather lengthy vector algebra manipulations. The derivation
of the momentum equation based on the standard field theory formalism is
given in Section \ref{sschrodinger}.

\section{Derivation of non-relativistic point dynamics for localized
solutions of NLS equation\label{SNLSpoint}}

In this section we consider solutions of NLS equation (\ref{NLS}) that are
localized around a trajectory in the three dimensional space $\mathbb{R}^{3}$
and find the necessary condition for such a trajectory which coincides with
Newton's law of motion.

\subsection{Trajectories and their neighborhoods}

The first step in introducing the concept of concentration to a trajectory
is to describe the class of trajectories.

\begin{definition}[trajectory]
\label{Dtrconc} A trajectory $\mathbf{\hat{r}}\left( t\right) $, $T_{-}\leq
t\leq T_{+}$, is a twice continuously differentiable function with values in 
$\mathbb{R}^{3}$ satisfying 
\begin{equation}
\left\vert \partial _{t}\mathbf{\hat{r}}\left( t\right) \right\vert \leq C,%
\text{\quad }\left\vert \partial _{t}^{2}\mathbf{\hat{r}}\left( t\right)
\right\vert \leq C\text{ \ for\ }T_{-}\leq t\leq T_{+}.  \label{trbound}
\end{equation}
\end{definition}

\ Being given a trajectory $\mathbf{\hat{r}}\left( t\right) $, we consider a
family of neighborhoods contracting to it. Namely, we introduce a ball of
radius $R$ centered at $\mathbf{x}=\mathbf{\hat{r}}\left( t\right) $:\ 
\begin{equation}
\Omega \left( \mathbf{\hat{r}}\left( t\right) ,R\right) =\left\{ \mathbf{x}%
:\left\vert \mathbf{x}-\mathbf{\hat{r}}\left( t\right) \right\vert ^{2}\leq
R^{2}\right\} \subset \mathbb{R}^{3},\text{\quad }R>0.  \label{Omball}
\end{equation}

\begin{definition}[concentrating neighborhoods]
\label{Dconcne}\emph{\ Concentrating neighborhoods} $\hat{\Omega}\left( 
\mathbf{\hat{r}}\left( t\right) ,R_{n}\right) \subset \left[ T_{-},T_{+}%
\right] \times \mathbb{R}^{3}$ of a trajectory $\mathbf{\hat{r}}\left(
t\right) $ are\ defined as a family of tubular domains 
\begin{equation}
\hat{\Omega}\left( \mathbf{\hat{r}}\left( t\right) ,R_{n}\right) =\left\{
\left( t,\mathbf{x}\right) :T_{-}\leq t\leq T_{+},\text{\quad }\left\vert 
\mathbf{x}-\mathbf{\hat{r}}\left( t\right) \right\vert ^{2}\leq
R_{n}^{2}\right\}  \label{Omhat}
\end{equation}%
where $R_{n}$ satisfy the contraction condition:%
\begin{equation}
R_{n}\rightarrow 0\text{\ as\ }n\rightarrow \infty .  \label{Rnto0}
\end{equation}%
The cross-section of the tubular domain at a fixed $t$ is given by (\ref%
{Omball}) and is denoted by $\Omega _{n}$:%
\begin{equation}
\Omega _{n}=\Omega _{n}\left( t\right) =\Omega \left( \mathbf{\hat{r}}%
(t),R_{n}\right) \subset \mathbb{R}^{3}.  \label{Omn}
\end{equation}
\end{definition}

\subsection{Localized NLS equations\label{SlocNLS}}

Let us consider the NLS equation (\ref{NLS}) in a neighborhood of the
trajectory $\mathbf{\hat{r}}(t)$. We remind that $m$ is the mass parameter, $%
q$ is the value of the charge, $\chi $ is a parameter similar to the Planck
constant, $\mathrm{c}$ is the speed of light all of which are fixed. The
size parameter $a$ and potentials $\varphi \left( t,\mathbf{x}\right) ,$ $%
\mathbf{A}\left( t,\mathbf{x}\right) $ form a sequence. We consider the NLS
in a shrinking vicinity of the trajectory and make certain regularity
assumptions on behavior of its coefficients. In the definitions below we use
the following notations:%
\begin{equation}
\partial _{0}=\mathrm{c}^{-1}\partial _{t},  \label{dt0}
\end{equation}%
\begin{gather}
\nabla _{\mathbf{x}}\varphi =\nabla \varphi =\left( \partial _{1}\varphi
,\partial _{2}\varphi ,\partial _{3}\varphi \right) ,\text{\quad }\left\vert
\nabla _{\mathbf{x}}\varphi \right\vert ^{2}=\left\vert \nabla \varphi
\right\vert ^{2}=\left\vert \partial _{1}\varphi \right\vert ^{2}+\left\vert
\partial _{2}\varphi \right\vert ^{2}+\left\vert \partial _{3}\varphi
\right\vert ^{2},  \label{grad0} \\
\nabla _{0,\mathbf{x}}\varphi =\left( \partial _{0}\varphi ,\partial
_{1}\varphi ,\partial _{2}\varphi ,\partial _{3}\varphi \right) ,\text{\quad 
}\left\vert \nabla _{0,\mathbf{x}}\varphi \right\vert ^{2}=\left\vert
\partial _{0}\varphi \right\vert ^{2}+\left\vert \partial _{1}\varphi
\right\vert ^{2}+\left\vert \partial _{2}\varphi \right\vert ^{2}+\left\vert
\partial _{3}\varphi \right\vert ^{2}.  \notag
\end{gather}%
Now we formulate the continuity assumtions we impose on the nonlinearity $%
G_{1}^{\prime }\ \ $in the NLS equation (\ref{NLS}).

\begin{condition}
\label{Dregnon}The real-valued function $G_{1}^{\prime }\left( s\right) $ is
continuous for $s>0$. It coincides with the derivative of the potential $%
G_{1}\left( s\right) $ which is \ differentiable for $s>0$ and continuous
for $s\geq 0$. We assume that \ the function $G_{1}\left( \psi ^{\ast }\psi
\right) \ $\ of the complex variable $\psi \in \mathbb{C}$ is \
differentiable for all $\psi $ and its differential\ has the form 
\begin{equation*}
dG_{1}\left( \psi ^{\ast }\psi \right) =g\left( \psi \right) d\psi ^{\ast
}+g^{\ast }\left( \psi \right) d\psi
\end{equation*}%
where 
\begin{equation*}
g\left( \psi \right) =\left\{ 
\begin{array}{c}
G_{1}^{\prime }\left( \psi ^{\ast }\psi \right) \psi \text{ \ for \ }\psi
\in \mathbb{C},\psi \neq 0, \\ 
0\text{ \ for \ }\psi =0%
\end{array}%
\right. ,
\end{equation*}
\ and $g\left( \psi \right) $ \ is continuous for \ $\psi \in \mathbb{C}$. \
\ 
\end{condition}

Note that the above condition allows a mild singularity of $G_{1}^{\prime }$
at zero, for example the logarithmic nonlinearity (\ref{Gaussg}) satisfies
this condition.

\begin{definition}[localized NLS equations]
\label{DNLSseq} Let $\mathbf{\hat{r}}(t)$ be a trajectory with its
concentrating neighborhoods $\hat{\Omega}\left( \mathbf{\hat{r}}%
,R_{n}\right) $, and let $a=a_{n}$, $\varphi =\varphi _{n},$\ $\mathbf{A}=%
\mathbf{A}_{n}$ \ be a sequence of parameters and potentials entering the
NLS equation (\ref{NLS}). We say that NLS equations are localized in $\hat{%
\Omega}\left( \mathbf{\hat{r}},R_{n}\right) $ if the following conditions
are satisfied. The parameters $R_{n}$ and $a_{n}$ as $n\rightarrow \infty $
become vanishingly small, that is\ 
\begin{equation}
a=a_{n}\rightarrow 0,\quad R_{n}\rightarrow 0  \label{anto0s}
\end{equation}%
and the ratio $\theta =R_{n}/a$\ grows to infinity: 
\begin{equation}
\theta _{n}=\frac{R_{n}}{a_{n}}\rightarrow \infty \ \text{as\ }n\rightarrow
\infty .  \label{theninfs}
\end{equation}%
We introduce at the trajectory$\ \mathbf{\hat{r}}(t)$ the limit potentials $%
\varphi _{\infty }\left( t,\mathbf{x}\right) $\ $\mathbf{A}_{\infty }\left(
t,\mathbf{x}\right) $ \ \ which are linear \ in $\mathbf{x}$ and are written
in the form%
\begin{equation}
\varphi _{\infty }\left( t,\mathbf{x}\right) =\varphi _{\infty }\left(
t\right) +\left( \mathbf{x}-\mathbf{\hat{r}}\right) \cdot \nabla \varphi
_{\infty }\left( t\right) ,  \label{fiinfs}
\end{equation}%
\ 
\begin{equation}
\mathbf{A}_{\infty }\left( t,\mathbf{x}\right) =\mathbf{A}_{\infty }\left(
t\right) +\left( \mathbf{x}-\mathbf{\hat{r}}\right) \cdot \nabla \mathbf{A}%
_{\infty }\left( t\right) ,  \label{ainfs}
\end{equation}%
where the coefficients $\varphi _{\infty },\nabla \varphi _{\infty },\mathbf{%
A}_{\infty },\nabla \mathbf{A}_{\infty }$ satisfy the following boundedness
conditions: 
\begin{equation}
\left\vert \varphi _{\infty }\right\vert \leq C,\text{\ }\left\vert \nabla
_{0,\mathbf{x}}\varphi _{\infty }\right\vert \leq C\text{ for}\;T_{-}\leq
t\leq T_{+},  \label{fihatbs}
\end{equation}%
\begin{equation}
\left\vert \mathbf{A}_{\infty }\right\vert \leq C,\text{\ }\left\vert \nabla
_{0,\mathbf{x}}\mathbf{A}_{\infty }\right\vert \leq C\text{ for}\;T_{-}\leq
t\leq T_{+}.  \label{ahatbs}
\end{equation}%
\ We suppose the EM\ potentials $\varphi _{n}\left( t,\mathbf{x}\right) ,%
\mathbf{A}_{n}\left( t,\mathbf{x}\right) $ to be twice continuously
differentiable in $\hat{\Omega}\left( \mathbf{\hat{r}},R_{n}\right) $. The
potentials $\varphi _{n}\left( t,\mathbf{x}\right) ,\mathbf{A}_{n}\left( t,%
\mathbf{x}\right) $ locally converge to the limit linear potentials $\varphi
_{\infty }\left( t,\mathbf{x}\right) $, $\ \mathbf{A}_{\infty }\left( t,%
\mathbf{x}\right) $, namely they satisfy the following relations:

\begin{enumerate}
\item[(i)] convergence\ as\ $n\rightarrow \infty $: 
\begin{equation}
\max_{T_{-}\leq t\leq T_{+},x\in \Omega _{n}}(|\varphi _{n}\left( t,\mathbf{x%
}\right) -\varphi _{\infty }\left( t,\mathbf{x}\right) \mathbf{|}+|\nabla
_{0,x}\varphi _{n}\left( t,\mathbf{x}\right) -\nabla _{0,x}\varphi _{\infty
}\left( t,\mathbf{x}\right) \mathbf{\mathbf{\mathbf{|}})}\rightarrow 0,
\label{Alocr0}
\end{equation}%
\ 
\begin{equation}
\max_{T_{-}\leq t\leq T_{+},x\in \Omega _{n}}(|\mathbf{A}_{n}\left( t,%
\mathbf{x}\right) -\mathbf{A}_{\infty }\left( t,\mathbf{x}\right) \mathbf{|}%
+|\nabla _{0,x}\mathbf{A}_{n}\left( t,\mathbf{x}\right) -\nabla _{0,x}%
\mathbf{A}_{\infty }\left( t,\mathbf{x}\right) \mathbf{\mathbf{\mathbf{|}})}%
\rightarrow 0;  \label{Alocr2}
\end{equation}

\item[(ii)] uniform in $n$ estimates: 
\begin{equation}
|\varphi _{n}\left( t,\mathbf{x}\right) \mathbf{|}\leq C\text{, \ \ }\mathbf{%
|}\nabla _{0,\mathbf{x}}\varphi _{n}\left( t,\mathbf{x}\right) \mathbf{%
\mathbf{|}}\leq C\text{ \ \ \ for }\;\left( t,\mathbf{x}\right) \in \hat{%
\Omega}\left( \mathbf{\hat{r}},R_{n}\right) ,  \label{Abounds0}
\end{equation}%
\begin{equation}
|\mathbf{A}_{n}\left( t,\mathbf{x}\right) \mathbf{|}\leq C\text{, \ }\mathbf{%
|}\nabla _{0,\mathbf{x}}\mathbf{A}_{n}\left( t,\mathbf{x}\right) \mathbf{%
\mathbf{|}}\leq C\text{ \ \ \ for }\;\left( t,\mathbf{x}\right) \in \hat{%
\Omega}\left( \mathbf{\hat{r}},R_{n}\right) .  \label{Abounds}
\end{equation}%
The limit EM fields $\mathbf{E}_{\infty }$,$\mathbf{B}_{\infty }$ at the
trajectory are defined in terms of the linear potentials (\ref{fiinfs}), (%
\ref{ainfs}) by (\ref{maxw3a}), namely 
\begin{equation}
\mathbf{E}_{\infty }=-\nabla \varphi _{\infty }\left( t,\mathbf{\hat{r}}%
\right) -\frac{1}{\mathrm{c}}\partial _{t}\mathbf{A}_{\infty }\left( t,%
\mathbf{\hat{r}}\right) ,\qquad \mathbf{B}_{\infty }=\nabla \times \mathbf{A}%
_{\infty }\left( t,\mathbf{\hat{r}}\right) .  \label{ebinf}
\end{equation}
\end{enumerate}
\end{definition}

Note that according to (\ref{Alocr0}), (\ref{Alocr2}) 
\begin{equation}
\mathbf{E}=\mathbf{E}_{n}\rightarrow \mathbf{E}_{\infty },\qquad \mathbf{B}=%
\mathbf{B}_{n}\rightarrow \mathbf{B}_{\infty }  \label{ebinf0a}
\end{equation}%
in $\hat{\Omega}\left( \mathbf{\hat{r}}(t),R_{n}\right) $.

Throughout this paper we denote constants which do not depend on $n$ by the
letter $C\ $with\ different\ indices. Sometimes the same letter $C$ with the
same indices may denote in different formulas different constants. Below we
often omit index $n$ in $a_{n}$, $\varphi _{n}$\ etc.

The most important case where we apply the above definition is described in
the following example.

\begin{example}
\label{EconcNLS}If the potentials $\varphi _{n},\mathbf{A}_{n}$\ are the
restrictions of fixed twice continuously differentiable potentials $\varphi ,%
\mathbf{A}$ to the domain $\hat{\Omega}\left( \mathbf{\hat{r}}%
(t),R_{n}\right) $, then $\varphi _{\infty ,}\mathbf{A}_{n}$ \ is the linear
part of $\varphi ,\mathbf{A}$ at \ $\mathbf{\hat{r}}$ and \ conditions (\ref%
{fihatbs})-(\ref{Abounds}) are satisfied with 
\begin{eqnarray}
\varphi _{\infty }\left( t,\mathbf{x}\right) &=&\varphi \left( t,\mathbf{%
\hat{r}}\right) +\left( \mathbf{x}-\mathbf{\hat{r}}\right) \cdot \nabla
\varphi \left( t,\mathbf{\hat{r}}\right) ,  \label{fiarestr} \\
\mathbf{A}_{\infty }\left( t,\mathbf{x}\right) &=&\mathbf{A}\left( t,\mathbf{%
\hat{r}}\right) +\left( \mathbf{x}-\mathbf{\hat{r}}\right) \cdot \nabla 
\mathbf{A}\left( t,\mathbf{\hat{r}}\right) ,  \notag
\end{eqnarray}%
that is the coefficients in (\ref{fiinfs}), (\ref{ainfs}) are defined as
follows: 
\begin{equation}
\varphi _{\infty }\left( t\right) =\varphi \left( t,\mathbf{\hat{r}}\left(
t\right) \right) ,\qquad \nabla \mathbf{A}_{\infty }\left( t\right) =\nabla 
\mathbf{A}\left( t,\mathbf{\hat{r}}\right) ,  \label{fiinf0}
\end{equation}%
and the EM fields $\mathbf{E}_{\infty },\mathbf{B}_{\infty }$ in (\ref{ebinf}%
) are directly expressed in terms of $\varphi ,\mathbf{A}$ by (\ref{maxw3a}%
), namely 
\begin{equation}
\mathbf{E}_{\infty }=-\nabla \varphi \left( t,\mathbf{\hat{r}}\right) -\frac{%
1}{\mathrm{c}}\partial _{t}\mathbf{A}\left( t,\mathbf{\hat{r}}\right)
,\qquad \mathbf{B}_{\infty }=\nabla \times \mathbf{A}\left( t,\mathbf{\hat{r}%
}\right) .  \label{ebinf0}
\end{equation}
\end{example}

We introduce the local value of the charge restricted to domain $\Omega
\left( \mathbf{\hat{r}}(t),R_{n}\right) $ by the formula 
\begin{equation}
\bar{\rho}_{n}\left( t\right) =\int_{\Omega \left( \mathbf{\hat{r}}%
(t),R_{n}\right) }\rho _{n}\,\mathrm{d}\mathbf{x}=\int_{\Omega \left( 
\mathbf{\hat{r}}(t),R_{n}\right) }q\left\vert \psi _{n}\right\vert ^{2}\,%
\mathrm{d}\mathbf{x},  \label{Endef0s}
\end{equation}%
with $\rho _{n}\left( t,\mathbf{x}\right) $ being the charge density defined
by (\ref{schr8}), and we call $\bar{\rho}_{n}$ adjacent charge value.

\begin{definition}[concentrating solutions]
\label{Dlocconverges}Let $\mathbf{\hat{r}}(t)$ be a trajectory. We say that
solutions $\psi $ to\ the NLS equations concentrate\ at\ the trajectory $%
\mathbf{\hat{r}}(t)$ if Condition \ref{Dregnon} is satisfied and the
following conditions are fulfilled. First, a sequence of concentrating
neighborhoods $\hat{\Omega}\left( \mathbf{\hat{r}},R_{n}\right) $,
parameters $a=a_{n}\ $\ and potentials$\ \varphi =\varphi _{n},\ \mathbf{A=}%
\ \mathbf{A}_{n}$ are selected as in Definition \ref{DNLSseq}. Second, there
exists a sequence of functions $\psi =\psi _{n}\ $which\ are twice
continuously differentiable,\ that is $\psi _{n}\in C^{2}\left( \hat{\Omega}%
\left( \mathbf{\hat{r}},R_{n}\right) \right) ,$ and such that the charge
density $\rho =q\left\vert \psi _{n}\right\vert ^{2}$ and the momentum
density $\mathbf{P}$ defined by (\ref{schr5}) for this sequence satisfy the
following conditions:

\begin{enumerate}
\item[(i)] every function $\psi _{n}$ is a solution to the NLS equation (\ref%
{NLS}) in $\hat{\Omega}\left( \mathbf{\hat{r}},R_{n}\right) $;

\item[(ii)] the momentum density $\mathbf{P}$ \ defined by (\ref{schr5}) for
this sequence is such that the following integrals are bounded:%
\begin{equation}
\left\vert \int_{\Omega _{n}}\mathbf{P}_{n}\left( t\right) \,\mathrm{d}%
\mathbf{x}\right\vert \leq C;  \label{locpsibounds}
\end{equation}

\item[(iii)] the local charge value defined by (\ref{Endef0s}) is bounded
from above and below\ for sufficiently large $n$: 
\begin{equation}
C\geq \bar{\rho}_{n}\left( t\right) \geq c_{0}>0\text{ for}\;n\geq
n_{0},\;T_{-}\leq t\leq T_{+};  \label{Egrcs}
\end{equation}

\item[(iv)] there exists $t_{0}\in \left[ T_{-},T_{+}\right] $ such that the
sequence of local charge values\ converges: 
\begin{equation}
\lim_{n\rightarrow \infty }\bar{\rho}_{n}\left( t_{0}\right) =\bar{\rho}%
_{\infty }.  \label{rconv}
\end{equation}

\item[(v)] We also impose conditions on the following surface integrals over
the spheres $\partial \Omega _{n}=\left\{ \left\vert \mathbf{x}-\mathbf{\hat{%
r}}\right\vert =R_{n}\right\} $: 
\begin{equation}
Q_{0}=\int_{t_{0}}^{t}\int_{\partial \Omega _{n}}\mathbf{\bar{n}}_{i}T^{ij}\,%
\mathrm{d}\sigma \mathrm{d}t^{\prime },  \label{QP0}
\end{equation}%
where $T^{ij}$ are given by (\ref{emfr10}), (\ref{Tij}) and we make
summation over repeating indices;\ 
\begin{equation}
Q_{01}=\int_{t_{0}}^{t}\int_{\partial \Omega _{n}}\mathbf{P\hat{v}}\cdot 
\mathbf{\bar{n}}\mathrm{d}\sigma \mathrm{d}t^{\prime },  \label{Q01a}
\end{equation}%
where $\mathbf{P}$ is defined by (\ref{schr5});%
\begin{equation}
Q_{20}=\mathbf{-}\int_{\partial \Omega _{n}}\left( \mathbf{x}-\mathbf{r}%
\right) \mathbf{\hat{v}}\cdot \mathbf{\bar{n}}\rho \,\mathrm{d}\sigma ,
\label{Q20}
\end{equation}%
where $\rho $ is defined by (\ref{schr8});%
\begin{equation}
Q_{22}=\int_{\partial \Omega _{n}}\left( \mathbf{x-r}\right) \mathbf{n}\cdot 
\mathbf{J}\mathrm{d}\mathbf{x},  \label{Q22}
\end{equation}%
where $\mathbf{J}$ is given by (\ref{schr8a}); and 
\begin{equation}
Q_{23}=\int_{t_{0}}^{t}\int_{\partial \Omega _{n}}\mathbf{\hat{v}\cdot n}%
\rho \mathrm{d}\mathbf{x}\mathrm{d}t^{\prime
}-\int_{t_{0}}^{t}\int_{\partial \Omega _{n}}\mathbf{n}\cdot \mathbf{J}%
\mathrm{d}\mathbf{x}\mathrm{d}t^{\prime }.  \label{Q23}
\end{equation}%
The integrals defined above are assumed to satisfy the following limit
relations uniformly on the time interval $\left[ T_{-},T_{+}\right] $:%
\begin{equation}
Q_{0}\rightarrow 0,  \label{Q0Tto0}
\end{equation}%
\begin{equation}
Q_{01}\rightarrow 0,  \label{Q01Pto0}
\end{equation}%
\begin{equation}
Q_{20}\rightarrow 0,  \label{Q20to0}
\end{equation}%
\begin{equation}
Q_{22}\rightarrow 0\mathbf{,}  \label{Q22to0}
\end{equation}%
\begin{equation}
Q_{23}\rightarrow 0.  \label{Q23to0}
\end{equation}

\item[(vi)] We introduce the following integrals over $\Omega _{n}$ \ with
vanishing at $\mathbf{\hat{r}}$ weights \ 
\begin{equation}
Q_{30}=\int_{\Omega _{n}}\left( \mathbf{E}-\mathbf{E}_{\infty }\right) \rho 
\mathrm{d}\mathbf{x},  \label{Q30}
\end{equation}%
\begin{equation}
Q_{31}=\int_{\Omega _{n}}\frac{1}{\mathrm{c}}\mathbf{J}\times \left( \mathbf{%
B}-\mathbf{B}_{\infty }\right) \mathrm{d}\mathbf{x},  \label{Q31}
\end{equation}%
and assume that 
\begin{equation}
\int_{t_{0}}^{t}\left( Q_{30}+Q_{31}\right) \mathrm{d}t^{\prime }\rightarrow
0  \label{Q3to0}
\end{equation}%
uniformly on the time interval $\left[ T_{-},T_{+}\right] $. If all the
above conditions are fulfilled, we call $\mathbf{\hat{r}}(t)$ a
concentration trajectory \ of the NLS equation (\ref{NLS}).
\end{enumerate}
\end{definition}

Obviously conditions (ii)-(iv) and (vi) provide boundedness and convergence
of certain volume integrals over $\Omega _{n}$, and condition (v) provides
asymptotical vanishing of surface integrals over the boundary. Notice that
condition (ii) provides for the boundedness of the momentum over domain $%
\Omega _{n}$, condition, (v) provides for a proper confinement of $\psi $ to 
$\Omega _{n}$ and estimate from below in condition (iii) ensures that\ the
sequence is non-trivial. According to (\ref{Egrcs}), $\bar{\rho}_{n}\left(
t_{0}\right) $ is a bounded sequence, consequently it always contains a
converging subsequence. Hence condition (iv) is not really an additional
constraint but rather it assumes that such a subsequence is selected. The
choice of a particular subsequence limit $\ \bar{\rho}_{\infty }$ is
discussed in Remark \ref{Runiqs}. This condition describes the amount of
charge which concentrates at the trajectory at the time $t_{0}$. \ Condition
(i) can be relaxed \ and replaced by the assumption that $\psi _{n}$ is an
asymptotic solution, see Definition \ref{Dconcas} \ for details. We could
also allow parameters $\chi ,m,q$ to form sequences and depend on $n$, but
for simplicity in this paper we assume them fixed.

\emph{The wave-corpuscles constructed in Section \ref{swcpres} provide a
non-trivial example of solutions\ which concentrate at trajectories of
accelerating charges}.

\subsection{Properties of concentrating solutions of NLS\label{Sprop}}

We define the adjacent charge center $\mathbf{r}_{n}$ by the formula 
\begin{equation}
\mathbf{r}_{n}\left( t\right) =\frac{1}{\bar{\rho}_{n}}\int_{\Omega \left( 
\mathbf{\hat{r}}(t),R_{n}\right) }\mathbf{x}\rho _{n}\,\mathrm{d}\mathbf{x}.
\label{rnoms}
\end{equation}%
Since $\psi _{a}\left( t,\mathbf{x}\right) $ is a function of class $C^{1}$
with respect to $\left( t,\mathbf{x}\right) $, and $\mathbf{\hat{r}}(t)$ is
differentiable, the vector $\mathbf{r}\left( t\right) $ is a differentiable
function of time, and we denote by $\mathbf{v}$\ the velocity of the
adjacent charge center: 
\begin{equation}
\mathbf{v}=\mathbf{v}_{n}\left( t\right) =\partial _{t}\mathbf{r.}
\end{equation}%
Below we often make use of the following elementary identity:%
\begin{equation}
\int_{\Omega _{n}}\partial _{t}f\left( t,\mathbf{x}\right) \,\mathrm{d}%
\mathbf{x}=\partial _{t}\int_{\Omega _{n}}f\left( t,\mathbf{x}\right) \,%
\mathrm{d}\mathbf{x}-\int_{\partial \Omega _{n}}f\left( t,\mathbf{x}\right) 
\mathbf{\hat{v}}\cdot \mathbf{\bar{n}}\,\mathrm{d}\sigma ,  \label{eldt}
\end{equation}%
where $\mathbf{\bar{n}}$ is the external normal\ to\ $\partial \Omega _{n}$, 
$\mathbf{\hat{v}}=\partial _{t}\mathbf{\hat{r}}$.\ 

\begin{lemma}
\label{Lcentrcs}Let charge densities $\rho _{n}$ satisfy (\ref{Egrcs}). Then
the adjacent ergocenters $\mathbf{r}_{n}\left( t\right) $ of the solutions
converge to\ $\mathbf{\hat{r}}(t)$\ uniformly on the time interval $\left[
T_{-},T_{+}\right] $.
\end{lemma}

\begin{proof}
By (\ref{rnoms}) 
\begin{equation}
\int_{\Omega _{n}}\left( \mathbf{x-r}\right) \rho _{n}\,\mathrm{d}\mathbf{x}%
=0,
\end{equation}%
and according to (\ref{Egrcs})%
\begin{equation}
\left\vert \int_{\Omega \left( \mathbf{\hat{r}}(t),R_{n}\right) }\left( 
\mathbf{x-\hat{r}}\right) \rho _{n}\,\mathrm{d}\mathbf{x}\right\vert \leq
R_{n}\int_{\Omega \left( \mathbf{0},R_{n}\right) }\,\rho _{n}\,\mathrm{d}%
\mathbf{y}\rightarrow 0.
\end{equation}%
Therefore%
\begin{equation*}
\left( \mathbf{r-\hat{r}}\right) \bar{\rho}=\int_{\Omega _{n}}\left( \mathbf{%
x-\hat{r}}\right) \rho _{n}\,\mathrm{d}\mathbf{x}-\int_{\Omega _{n}}\left( 
\mathbf{x-r}\right) \rho _{n}\,\mathrm{d}\mathbf{x}\rightarrow 0.
\end{equation*}%
\ Using (\ref{Egrcs}) we conclude that 
\begin{equation}
\left\vert \mathbf{\hat{r}}-\mathbf{r}_{n}\right\vert \rightarrow 0
\label{rminrns}
\end{equation}%
uniformly on $\left[ T_{-},T_{+}\right] $.
\end{proof}

\begin{lemma}
\label{Ladjch} Let the current $\mathbf{J}$ and the charge density $\rho $
in (\ref{schr9}) are such that $Q_{23}$ defined by (\ref{Q23}) satisfies (%
\ref{Q23to0}). Then the local charge values converge uniformly on $\left[
T_{-},T_{+}\right] $\ to a constant: 
\begin{equation}
\bar{\rho}_{n}\left( t\right) \rightarrow \bar{\rho}_{\infty }\text{\quad
as\quad }n\rightarrow \infty .  \label{rtrinf}
\end{equation}
\end{lemma}

\begin{proof}
Integrating the continuity equation \ we obtain 
\begin{equation}
\bar{\rho}\left( t\right) -\bar{\rho}\left( t_{0}\right)
-\int_{t_{0}}^{t}\int_{\partial \Omega _{n}}\mathbf{\hat{v}\cdot n}\rho \,%
\mathrm{d}\mathbf{x}\mathrm{d}t^{\prime }+\int_{t_{0}}^{t}\int_{\partial
\Omega _{n}}\mathbf{n}\cdot \mathbf{J}\,\mathrm{d}\mathbf{x}\mathrm{d}%
t^{\prime }=0.  \label{dtrb}
\end{equation}%
We use (\ref{Q23to0}) and obtain (\ref{rtrinf}).
\end{proof}

\begin{lemma}
\label{Lmomf}Assume that (\ref{locpsibounds}) \ holds. Then there is a
subsequence of the concentrating sequence of solutions of the NLS such that
\ 
\begin{equation}
\int_{\Omega \left( \mathbf{\hat{r}}(t),R_{n}\right) }\mathbf{P}_{n}\left(
t_{0}\right) \,\mathrm{d}\mathbf{x}\rightarrow \mathbf{p}_{\infty }\text{%
\quad as\quad }n\rightarrow \infty .  \label{pnt0}
\end{equation}%
Assume also that boundary integrals (\ref{QP0}) and (\ref{Q01a}) satisfy
assumptions (\ref{Q0Tto0}) and (\ref{Q01Pto0}). Then for this subsequence\ 
\begin{equation}
\int_{\Omega _{n}}\mathbf{P}_{n}\left( t\right) \,\mathrm{d}\mathbf{x}%
=\int_{t_{0}}^{t}\int_{\Omega _{n}}\mathbf{f}\,\mathrm{d}\mathbf{x}\mathrm{d}%
t^{\prime }+\mathbf{p}_{\infty }+Q_{00}  \label{intpt}
\end{equation}%
where 
\begin{equation}
Q_{00}\rightarrow 0\text{\quad as\quad }n\rightarrow \infty  \label{Q00to0}
\end{equation}%
uniformly on $\left[ T_{-},T_{+}\right] $.
\end{lemma}

\begin{proof}
According to (\ref{locpsibounds}) we can select a subsequence which has a
limit $\mathbf{p}_{\infty }$ \ and (\ref{pnt0}) holds. We use the momentum
equation (\ref{momNLS})%
\begin{equation}
\partial _{t}\mathbf{P}+\partial _{i}T^{ij}=\mathbf{f}.  \label{momNLSf}
\end{equation}%
Integrating it over $\Omega \left( \mathbf{\hat{r}}(t),R_{n}\right) =\Omega
_{n}$\ and then with respect to time, we obtain the equation%
\begin{equation}
\int_{t_{0}}^{t}\int_{\Omega _{n}}\partial _{t}\mathbf{P}\left( t^{\prime
}\right) \,\mathrm{d}\mathbf{x}\mathrm{d}t^{\prime
}-\int_{t_{0}}^{t}\int_{\Omega _{n}}\mathbf{f}\,\mathrm{d}\mathbf{x}\mathrm{d%
}t^{\prime }\,+Q_{0}=0  \label{momQ0}
\end{equation}%
where $\ Q_{0}$ is defined by (\ref{QP0}). Using (\ref{eldt}) \ we rewrite
the equation in the form 
\begin{equation}
\int_{\Omega _{n}}\mathbf{P}\left( t\right) \,\mathrm{d}\mathbf{x}%
-\int_{\Omega _{n}}\mathbf{P}\left( t_{0}\right) \,\mathrm{d}\mathbf{x}%
-Q_{01}\,-\int_{t_{0}}^{t}\int_{\Omega _{n}}\mathbf{f}\,\mathrm{d}\mathbf{x}%
\mathrm{d}t^{\prime }\,+Q_{0}=0,  \label{Peq}
\end{equation}%
where $Q_{01}$ is defined by (\ref{Q01a}), $\Omega _{n}=\Omega \left( 
\mathbf{\hat{r}}(t),R_{n}\right) $. Using (\ref{Q0Tto0}), (\ref{Q01Pto0})
and (\ref{pnt0}) we obtain (\ref{intpt}) and (\ref{Q00to0}) \ from (\ref{Peq}%
).
\end{proof}

\begin{lemma}
\label{LJlimv}In addition to the assumptions of Lemmas \ref{Lcentrcs}, \ref%
{Ladjch}, \ref{Lmomf} assume that \ the boundary integrals (\ref{Q20}) and (%
\ref{Q22}) vanish as $n\rightarrow \infty $, namely (\ref{Q20to0}) and (\ref%
{Q22to0}) are fulfilled uniformly on the time interval $\left[ T_{-},T_{+}%
\right] $. Then 
\begin{equation}
\int_{\Omega _{n}}\mathbf{J}\,\mathrm{d}\mathbf{x}=\mathbf{v}\bar{\rho}%
_{\infty }+Q_{200},  \label{j200}
\end{equation}%
\begin{equation}
\int_{\Omega _{n}}\mathbf{P}\,\mathrm{d}\mathbf{x}=m\mathbf{v}\frac{\bar{\rho%
}_{\infty }}{q}+\frac{m}{q}Q_{200},  \label{p200}
\end{equation}%
where $\mathbf{v}=\partial _{t}\mathbf{r}$ and%
\begin{equation}
Q_{200}\rightarrow 0  \label{Q2000}
\end{equation}%
\ uniformly on the time interval $\left[ T_{-},T_{+}\right] .$
\end{lemma}

\begin{proof}
According to the continuity equation, \ we obtain the identity 
\begin{equation*}
\int_{\Omega _{n}}\left( \mathbf{x-r}\right) \partial _{t}\rho \,\mathrm{d}%
\mathbf{x}+\int_{\Omega _{n}}\left( \mathbf{x-r}\right) \nabla \cdot \mathbf{%
J}\,\mathrm{d}\mathbf{x}=0.
\end{equation*}%
Using the commutation relation%
\begin{equation}
\partial _{j}\left( x_{i}f\right) -x_{i}\partial _{j}f=\delta _{ij}f
\label{comdx}
\end{equation}%
to transform the second integral, we obtain the following equation%
\begin{equation}
\int_{\Omega _{n}}\partial _{t}\left( \left( \mathbf{x-r}\right) \rho
\right) \,\mathrm{d}\mathbf{x}+\partial _{t}\mathbf{r}\int_{\partial \Omega
_{n}}\rho \,\mathrm{d}\mathbf{x}+\int_{\partial \Omega _{n}}\left( \mathbf{%
x-r}\right) \mathbf{n}\cdot \mathbf{J}\,\mathrm{d}\mathbf{x}=\int_{\Omega
_{n}}\mathbf{J}\,\mathrm{d}\mathbf{x.}  \label{Pintr0}
\end{equation}%
Using the definition of the charge center $\mathbf{r}$\ and (\ref{eldt}) we
infer that the first term in the above equation has the following form: 
\begin{equation}
\int_{\Omega _{n}}\partial _{t}\left( \left( \mathbf{x}-\mathbf{r}\right)
\rho \right) \,\mathrm{d}\mathbf{x=-}\int_{\partial \Omega _{n}}\left( 
\mathbf{x}-\mathbf{r}\right) \mathbf{\hat{v}}\cdot \mathbf{\bar{n}}\rho \,%
\mathrm{d}\sigma .  \label{dtxrh}
\end{equation}%
where $\mathbf{\hat{v}}=\partial _{t}\mathbf{\hat{r}}$. We can express then $%
\mathbf{v}=\partial _{t}\mathbf{r}$ from (\ref{Pintr0}), and using (\ref%
{locpsibounds}), (\ref{Q22}), (\ref{rminrns}) and (\ref{Egrcs}) can estimate
the integrals which enter (\ref{Pintr0}) concluding that 
\begin{equation}
\left\vert \mathbf{v}\right\vert \leq C\quad \text{for}\quad -T\leq t\leq T.
\label{dtrlec}
\end{equation}

We rewrite (\ref{Pintr0}) in the form 
\begin{equation}
\int_{\Omega _{n}}\mathbf{J}\mathrm{d}\,\mathbf{x}=\mathbf{v}\bar{\rho}%
_{\infty }+Q_{20}+Q_{21}+Q_{22},  \label{Pintr}
\end{equation}%
where $\bar{\rho}_{\infty }$ is the same as in (\ref{rconv}), 
\begin{equation*}
Q_{21}=\mathbf{v}\left( \int_{\Omega _{n}}\rho _{n}\,\mathrm{d}\mathbf{x}-%
\bar{\rho}_{\infty }\right) .
\end{equation*}%
According to (\ref{rconv}) and (\ref{dtrlec}) $Q_{21}\rightarrow 0$. Using (%
\ref{Q20}) and (\ref{Q22})\ we conclude that (\ref{j200}) and (\ref{Q2000})
hold with $Q_{200}=Q_{20}+Q_{21}+Q_{22}$ . Using (\ref{PmqJ}) we deduce (\ref%
{p200}) from (\ref{j200}).
\end{proof}

\subsection{Derivation of Newton's equation for the trajectory of
concentration\label{Sder}}

\begin{theorem}
\label{Tconctr}For a concentrating sequence of solutions of the NLS equation
the adjacent center velocities $\mathbf{v=v}_{n}=\mathbf{\partial }_{t}%
\mathbf{r}_{n}$ satisfy the equation 
\begin{equation}
m\mathbf{v}\frac{1}{q}\bar{\rho}_{\infty }=\int_{t_{0}}^{t}\left( \bar{\rho}%
_{\infty }\mathbf{E}_{\infty }\left( t^{\prime }\right) +\bar{\rho}_{\infty }%
\mathbf{v}\times \frac{1}{\mathrm{c}}\mathbf{B}_{\infty }\left( t^{\prime
}\right) \right) \,\mathrm{d}t^{\prime }+\mathbf{p}_{\infty }+Q_{6}
\label{NNLSQ}
\end{equation}%
where $\ Q_{6}\rightarrow 0$ \ uniformly on $\left[ T_{-},T_{+}\right] $, $%
\mathbf{p}_{\infty }$ is the same as in (\ref{pnt0}), and $\mathbf{E}%
_{\infty },\mathbf{B}_{\infty }$ are the same as in (\ref{ebinf}).
\end{theorem}

\begin{proof}
We substitute (\ref{p200}) into (\ref{intpt}) \ and obtain \ the following
equation:%
\begin{equation}
m\mathbf{v}\frac{1}{q}\bar{\rho}_{\infty }+\frac{m}{q}Q_{200}=%
\int_{t_{0}}^{t}\int_{\Omega _{n}}\mathbf{f}\,\mathrm{d}\mathbf{x}\mathrm{d}%
t^{\prime }+\mathbf{p}_{\infty }+Q_{00}.  \label{Peqa}
\end{equation}%
where the Lorentz force density $\mathbf{f}$ is given by (\ref{flors}). \ To
evaluate the terms involving the Lorentz force density, \ we use\ the limit
EM fields defined in accordance with (\ref{fiinfs}), (\ref{ainfs})\ by
formula (\ref{ebinf}). \ We obtain%
\begin{equation}
\int_{\Omega _{n}}\mathbf{f}\,\mathrm{d}\mathbf{x}=\mathbf{E}_{\infty
}\int_{\Omega _{n}}\rho \,\mathrm{d}\mathbf{x}+\int_{\Omega _{n}}\frac{1}{%
\mathrm{c}}\mathbf{J}\,\mathrm{d}\mathbf{x}\times \mathbf{B}_{\infty }+Q_{3}
\label{flor2}
\end{equation}%
where $\ Q_{3}=Q_{30}+Q_{31}$ \ is expressed in terms of (\ref{Q30}), (\ref%
{Q31}). Therefore, using (\ref{j200}) \ we obtain from (\ref{flor2}) that 
\begin{eqnarray}
\int_{\Omega _{n}}\mathbf{f}\,\mathrm{d}\mathbf{x} &=&\bar{\rho}_{\infty }%
\mathbf{E}_{\infty }+\bar{\rho}_{\infty }\mathbf{v}\times \frac{1}{\mathrm{c}%
}\mathbf{B}_{\infty }+Q_{3}+Q_{4}+Q_{5}, \\
Q_{4} &=&\frac{1}{\mathrm{c}}Q_{200}\times \mathbf{B}_{\infty },\quad Q_{5}=%
\mathbf{E}_{\infty }\left( \bar{\rho}_{n}-\bar{\rho}_{\infty }\right) .
\end{eqnarray}%
Using (\ref{Q3to0}), (\ref{rtrinf}), (\ref{Q2000}), (\ref{fihatbs}), and (%
\ref{ahatbs}),\ we obtain (\ref{NNLSQ}) with 
\begin{equation*}
Q_{6}=Q_{00}-\frac{m}{q}Q_{200}+\int_{t_{0}}^{t}\left(
Q_{3}+Q_{4}+Q_{5}\right) \mathrm{d}t^{\prime }.
\end{equation*}
\end{proof}

The following statement provides an explicit necessary condition for a
trajectory of concentration.

\begin{theorem}[Trajectory of concentration criterion]
\label{TconcNLS}Let solutions $\psi $ of\ the NLS\ equation (\ref{NLS})
concentrate at $\mathbf{\hat{r}}(t)$. Then the trajectory $\mathbf{\hat{r}}%
\left( t\right) $\ satisfies the equation 
\begin{equation}
\partial _{t}^{2}\mathbf{\hat{r}}=\mathbf{f}_{\infty }  \label{New2s}
\end{equation}%
with the Lorentz force $\mathbf{f}_{\infty }\left( t\right) $ expressed in
terms of the potentials (\ref{fiinfs}), (\ref{ainfs}) by the formula\ 
\begin{equation}
\mathbf{f}_{\infty }\left( t\right) =-q\mathbf{E}_{\infty }+\frac{q}{\mathrm{%
c}}\partial _{t}\mathbf{\hat{r}}\times \mathbf{B}_{\infty }.
\label{FLorinfs}
\end{equation}
\end{theorem}

\begin{proof}
The function $\mathbf{v}\left( t\right) $ can be considered as a solution of
the integral equation \ (\ref{NNLSQ}) on the interval $\left[ T_{-},T_{+}%
\right] $, and this equation is evidently linear. Since the term $%
Q_{6}\rightarrow 0\ $uniformly, the sequence $\mathbf{v}_{n}\left( t\right) $
of its solutions converges uniformly to the solution of the equation 
\begin{equation}
m\mathbf{v}_{\infty }\frac{1}{q}\bar{\rho}_{\infty }=\bar{\rho}_{\infty
}\int_{t_{0}}^{t}\left( \mathbf{E}_{\infty }\left( t^{\prime }\right) +%
\mathbf{v}_{\infty }\times \frac{1}{\mathrm{c}}\mathbf{B}_{\infty }\left(
t^{\prime }\right) \right) \,\mathrm{d}t^{\prime }+\mathbf{p}_{\infty }.
\label{vinfs}
\end{equation}%
Now we want to prove that $\mathbf{v}_{\infty }=\mathbf{\hat{v}}=\partial
_{t}\mathbf{\hat{r}}$. Note that according to Lemma \ref{Lcentrcs} 
\begin{equation*}
\int_{t_{0}}^{t}\mathbf{v}_{n}\,\mathrm{d}t^{\prime }=\mathbf{r}_{n}\left(
t\right) -\mathbf{r}_{n}\left( t_{0}\right) \rightarrow \mathbf{\hat{r}}%
\left( t\right) -\mathbf{\hat{r}}\left( t_{0}\right) ,
\end{equation*}%
and, hence, 
\begin{equation*}
\int_{t_{0}}^{t}\mathbf{v}_{\infty }\,\mathrm{d}t^{\prime }=\mathbf{\hat{r}}%
\left( t\right) -\mathbf{\hat{r}}\left( t_{0}\right) .
\end{equation*}%
The above identity implies $\partial _{r}\mathbf{\hat{r}}=\mathbf{\hat{v}}%
\left( t\right) =\mathbf{v}_{\infty }\left( t\right) $ and consequently $%
\mathbf{\hat{r}}\left( t\right) $ satisfies (\ref{vinfs}):%
\begin{equation}
m\partial _{t}\mathbf{\hat{r}\ }\frac{1}{q}\bar{\rho}_{\infty }=\bar{\rho}%
_{\infty }\int_{t_{0}}^{t}\left( \mathbf{E}_{\infty }\left( t^{\prime
}\right) +\partial _{r}\mathbf{\hat{r}\ }\times \frac{1}{\mathrm{c}}\mathbf{B%
}_{\infty }\left( t^{\prime }\right) \right) \,\mathrm{d}t^{\prime }+\mathbf{%
p}_{\infty }.
\end{equation}%
The derivative of the above equation yields (\ref{New2s}).
\end{proof}

As a corollary of Theorem \ref{TconcNLS} we obtain the following theorem
describing the whole class of trajectories of concentration of NLS equation (%
\ref{NLS}).

\begin{theorem}[Non-relativistic Newton's law]
\label{Cq0s}Assume that (i) potentials $\varphi \left( t,\mathbf{x}\right) $%
, $\mathbf{A}\left( t,\mathbf{x}\right) $ are defined and regular in a
domain $D\subset \mathbb{R}\times \mathbb{R}^{3}$; (ii)\ the trajectory $%
\left( t,\mathbf{\hat{r}}\left( t\right) \right) $ lies in this domain and
the limit potentials $\varphi _{\infty }$ , $\mathbf{A}_{\infty }$ are the
restriction of the potentials $\varphi $ and $\mathbf{A}$ as in (\ref%
{fiarestr}). Let EM fields \ $\mathbf{E}\left( t,\mathbf{x}\right) $, $%
\mathbf{B}\left( t,\mathbf{x}\right) $ be defined by the potentials as in
the formula (\ref{maxw3a}). Let solutions $\psi $ of\ the NLS\ equation (\ref%
{NLS}) concentrate at $\mathbf{\hat{r}}(t)$. Then equation (\ref{New2s}) for
the trajectory $\mathbf{\hat{r}}$ takes the form of Newton's law of motion
with the Lorentz force corresponding to the external EM fields $\mathbf{E}%
\left( t,\mathbf{x}\right) $ and $\mathbf{B}\left( t,\mathbf{x}\right) $,
that is 
\begin{equation}
m\partial _{t}^{2}\mathbf{\hat{r}}=q\mathbf{E}\left( t,\mathbf{\hat{r}}%
\right) +\frac{q}{\mathrm{c}}\partial _{t}\mathbf{\hat{r}}\times \mathbf{B}%
\left( t,\mathbf{\hat{r}}\right) .  \label{New2sa}
\end{equation}
\end{theorem}

Therefore, for the NLS equation with given potentials $\varphi $, $\mathbf{A}
$ any concentration trajectory must coincide with the solution of the
equation (\ref{New2sa}). For given potentials $\varphi $, $\mathbf{A}$ \ the
concentration trajectory through a point $\ \mathbf{\hat{r}}\left(
t_{0}\right) $ is uniquely determined by the velocity $\partial _{t}\mathbf{%
\hat{r}}\left( t_{0}\right) $. In particular, it does not depend on the
parameter $\chi $ which enters the NLS equation. All possible concentration
trajectories are solutions of the equation (\ref{New2sa}).

\begin{remark}
\label{Runiqs} Consider the situation of Example \ref{EconcNLS} and Theorem %
\ref{Cq0s}. The sequences $a_{n}$, $R_{n}$, $\theta _{n}$, $\varphi _{n}$, $%
\mathbf{A}_{n}$, $\psi _{n}$\ enter the definition of a concentrating
solution. If we take two different sequences which fit the definition, we
obtain the same equation (\ref{New2sa}). More than that, the\ trajectory $%
\mathbf{\hat{r}}\left( t\right) $ is uniquely defined by the initial data $%
\mathbf{\hat{r}}\left( t_{0}\right) $\textbf{, }$\partial _{t}\mathbf{\hat{r}%
}\left( t_{0}\right) $ and consequently it\ does not depend on the
particular sequence. Remarkably, the equation for the trajectory is
independent of the value of $\bar{\rho}_{\infty }$, and this property does
not hold in the relativistic case, see \cite{BF9}.
\end{remark}

\begin{remark}
We can modify the definition of concentration at a trajectory by allowing
parameter $\chi \ $to be not fixed but form a sequence. The statements of \
Theorem \ref{TconcNLS} and Theorem \ref{Cq0s} continue to hold in this case.
\end{remark}

\section{Wave-corpuscles in accelerated motion \label{swcpres}}

An example of concentrating solutions is provided by what we call
wave-corpuscles. We define the wave-corpuscle $\psi $ by the formula 
\begin{gather}
\psi \left( t,\mathbf{x}\right) =\mathrm{e}^{\mathrm{i}S}\mathring{\psi}%
\left( \left\vert \mathbf{y}\right\vert \right) ,  \label{psil00} \\
S=S\left( \mathbf{y},t\right) ,\quad \mathbf{y}=\mathbf{x}-\mathbf{r}\left(
t\right) ,  \label{psil00a}
\end{gather}%
where the form factor $\mathring{\psi}$ satisfies the steady-state equation (%
\ref{nop40}) and the normalization condition (\ref{norm}).

To give non-trivial examples of localized form factors, it is convenient to
start with the form factor $\mathring{\psi}\left( r\right) $ and to describe
the nonlinearity which produces the form factor as a solution of (\ref{nop40}%
). To define the nonlinearity, we impose first our requirements on the
ground state $\mathring{\psi}\left( r\right) $ of the charge distribution. A
ground state is a positive function $\mathring{\psi}\left( r\right) $, $%
r=\left\vert x\right\vert $, which is twice differentiable, satisfies the
charge normalization condition (\ref{norm}) and is monotonically decreasing: 
\begin{equation}
\partial _{r}\mathring{\psi}\left( r\right) <0\quad \text{for}\quad r>0.
\label{drpsile0}
\end{equation}%
If $\mathring{\psi}\left( r\right) $ is a ground state, we can determine the
nonlinearity $G^{\prime }$ from the following equation obtained from (\ref%
{nop40}): 
\begin{equation}
\nabla ^{2}\mathring{\psi}=G^{\prime }(|\mathring{\psi}|^{2})\mathring{\psi}.
\end{equation}%
For a radial $\mathring{\psi}$ we obtain then an expression for the
nonlinearity $G^{\prime }$: 
\begin{equation}
G^{\prime }\left( \mathring{\psi}^{2}\left( r\right) \right) =G^{\prime
}\left( \mathring{\psi}^{2}\left( r\right) \right) =\frac{(\nabla ^{2}%
\mathring{\psi})\left( r\right) }{\mathring{\psi}\left( r\right) }.
\label{gg1}
\end{equation}%
Since $\mathring{\psi}^{2}\left( r\right) $ is a \emph{monotonic} function,\
we can find its inverse $r=r\left( \psi ^{2}\right) ,$ yielding 
\begin{equation}
G^{\prime }\left( s\right) =\frac{\nabla ^{2}\mathring{\psi}\left( r\left(
s\right) \right) }{\mathring{\psi}\left( r\left( s\right) \right) },\quad 0=%
\mathring{\psi}^{2}\left( \infty \right) \leq s\leq \mathring{\psi}%
^{2}\left( 0\right) .  \label{intpsa}
\end{equation}%
Since $\mathring{\psi}\left( r\right) $ is smooth and $\partial _{r}%
\mathring{\psi}<0$, $G^{\prime }(|\psi |^{2})$ is smooth for $0<|\psi |^{2}<%
\mathring{\psi}^{2}\left( 0\right) $; we extend $G^{\prime }(s)$ for $s\geq 
\mathring{\psi}^{2}\left( 0\right) $ \ as a smooth function for all $s>0$.
To describe the localization of the ground state $\mathring{\psi}$, we
introduce an explicit dependence on the \emph{size parameter} $a>0$ as in (%
\ref{nrac5}):%
\begin{equation}
\mathring{\psi}\left( r\right) =\mathring{\psi}_{a}\left( r\right) =a^{-3/2}%
\mathring{\psi}_{1}\left( a^{-1}r\right) ,\quad r=\left\vert x\right\vert
\geq 0.
\end{equation}%
This corresponds to the dependence of the nonlinearity on the size parameter
\ described by (\ref{Gas}). \ Note that the antiderivative $G\left( s\right) 
$ is defined for $s\geq 0$ according to Condition \ref{Dregnon} and is given
by the formula. 
\begin{equation}
G\left( s\right) =\int_{0}^{s}G^{\prime }\left( s^{\prime }\right) \,\mathrm{%
d}s^{\prime }  \label{Gan}
\end{equation}

\begin{example}
\label{Egau}We define a \emph{Gaussian} form factor by the formula 
\begin{equation}
\mathring{\psi}\left( r\right) =C_{g}\mathrm{e}^{-r^{2}/2}  \label{Gaussp}
\end{equation}%
where $C_{g}$ is a normalization factor, $C_{g}=\pi ^{-3/4}$ if $\upsilon
_{0}=1$ in (\ref{norm}). Such a ground state is called \emph{gausson} in 
\cite{Bialynicki}, \cite{Bialynicki1}. Elementary computation shows that 
\begin{equation*}
\frac{\nabla ^{2}\mathring{\psi}\left( r\right) }{\mathring{\psi}\left(
r\right) }=r^{2}-3=-\ln \left( \mathring{\psi}^{2}\left( r\right)
/C_{g}^{2}\right) -3.
\end{equation*}%
Hence, we define the nonlinearity corresponding to the Gaussian by the
formula%
\begin{equation}
G^{\prime }\left( |\psi |^{2}\right) =-\ln \left( |\psi
|^{2}/C_{g}^{2}\right) -3,  \label{Gaussg}
\end{equation}%
and refer to it as the \emph{logarithmic nonlinearity. }The nonlinear
potential function has the form%
\begin{equation}
G\left( s\right) =\int_{0}^{s}G^{\prime }\left( s^{\prime }\right) \,\mathrm{%
d}s^{\prime }=-s\ln s+s\left( \ln \frac{1}{\pi ^{3/2}}-2\right) .
\label{g1gauss}
\end{equation}%
\emph{\ }Dependence on the size parameter $a>0$ is given by the formula 
\begin{equation}
G_{a}^{\prime }\left( |\psi |^{2}\right) =-a^{-2}\ln \left( a^{3}|\psi
|^{2}/C_{g}^{2}\right) -3a^{-2}.  \label{Gpa}
\end{equation}%
Note that according to Gross inequality the Gaussian provides the minimum of
energy%
\begin{equation*}
E=\int u\,\mathrm{d}\mathbf{x}
\end{equation*}%
subjected to the normalization condition (\ref{norm}), where the energy
density $u$ is defined by (\ref{emfr8}).
\end{example}

More examples of nonlinearities can be found in \cite{BF5}--\cite{BF9}, and
many facts from the theory of the NLS equations can be found in \cite%
{Cazenave03}, \cite{Sulem}. The NLS equations with logarithmic nonlinearity
are studied in \cite{CazenaveHaraux80}, \cite{Cazenave83}, \cite{Cazenave03}.

Below we find conditions on the external fields $\varphi $ and $\mathbf{A}$
which allow the wave-corpuscle of the form (\ref{psil00}) to preserve
exactly its shape $\left\vert \psi \right\vert $ in accelerated motion
governed by the NLS equation along a trajectory $\mathbf{r}\left( t\right) $%
. In particular, Newton's law of motion emerges as the necessary condition
for such a motion.

\subsection{Criterion for shape preservation}

It is convenient to rewrite the NLS equation (\ref{NLS}) in the moving frame
using a change of variables $\mathbf{x}-\mathbf{r}\left( t\right) =\mathbf{y}
$. In $y$-coordinates the NLS equation (\ref{NLS}) takes the form 
\begin{equation}
\chi \mathrm{i}\partial _{t}\psi -\chi \mathrm{i}\mathbf{v}\cdot \nabla \psi
-q\varphi \psi =\frac{\chi ^{2}}{2m}\left[ -\left( \nabla -\frac{\mathrm{i}q%
}{\chi c}\mathbf{A}\right) ^{2}\psi +G^{\prime }\left( \psi ^{\ast }\psi
\right) \psi \right] ,  \label{Schy}
\end{equation}%
where $\mathbf{v}$ is the center velocity: 
\begin{equation}
\mathbf{v}\left( t\right) =\partial _{t}\mathbf{r}\left( t\right) \mathbf{.}
\label{velp}
\end{equation}%
We substitute (\ref{psil00}) in (\ref{Schy}), and, canceling the factor $%
\mathrm{e}^{\mathrm{i}S},\ $ we arrive to the following equivalent equation
for the phase $S$: 
\begin{gather}
-\chi \mathring{\psi}\partial _{t}S-\chi \mathrm{i}\mathbf{v}\cdot \nabla 
\mathring{\psi}+\chi \mathbf{v}\cdot \nabla S\mathring{\psi}-q\varphi 
\mathring{\psi}=  \label{Schy1} \\
=\frac{\chi ^{2}}{2m}\left[ -\left( \nabla -\frac{\mathrm{i}q}{\chi c}%
\mathbf{A}+\mathrm{i}\nabla S\right) ^{2}\mathring{\psi}+G^{\prime }\left( 
\mathring{\psi}^{2}\right) \mathring{\psi}\right] .  \notag
\end{gather}%
Expanding $\left( \nabla -\frac{\mathrm{i}q\mathbf{A}}{\chi c}+\mathrm{i}%
\nabla S\right) ^{2}\mathring{\psi}$ and using (\ref{nop40}) to exclude the
nonlinearity $G^{\prime }$,\ we rewrite (\ref{Schy1}) in the form 
\begin{gather}
-\mathring{\psi}\chi \partial _{t}S-\chi \mathrm{i}\mathbf{v}\cdot \nabla 
\mathring{\psi}+\chi \mathbf{v}\cdot \nabla S\mathring{\psi}-q\varphi 
\mathring{\psi}=  \label{Schy2} \\
=\frac{\chi ^{2}}{2m}\left[ 2\left( \frac{\mathrm{i}q}{\chi c}\mathbf{A}-%
\mathrm{i}\nabla S\right) \nabla \mathring{\psi}-\mathring{\psi}\nabla \cdot
\left( \mathrm{i}\nabla S-\frac{\mathrm{i}q}{\chi c}\mathbf{A}\right)
+\left( \nabla S-\frac{q}{\chi c}\mathbf{A}\right) ^{2}\mathring{\psi}\right]
.  \notag
\end{gather}%
Every term in the above equation has either the factor $\mathring{\psi}$ or $%
\nabla \mathring{\psi}$. Collecting terms at $\nabla \mathring{\psi}$ \ and $%
\mathring{\psi}$ respectively, we obtain two equations: 
\begin{equation}
\left( \chi \nabla S-\frac{q}{c}\mathbf{A}-m\mathbf{v}\right) \cdot \nabla 
\mathring{\psi}=0,  \label{grp0}
\end{equation}%
\begin{equation}
-\chi \partial _{t}S+\chi \mathbf{v}\cdot \nabla S-q\varphi =-\mathrm{i}%
\frac{\chi }{2m}\nabla \cdot \left( \chi \nabla S-\frac{q}{c}\mathbf{A}%
\right) +\frac{1}{2m}\left( \chi \nabla S-\frac{q}{c}\mathbf{A}\right) ^{2}.
\label{p1}
\end{equation}

Now we would like to find conditions on the potentials $\varphi ,\mathbf{A}$
under which there exists the phase $S$ which satisfies the above system of
equations (\ref{grp0}), (\ref{p1}). Such a phase provides for the
wave-corpuscle solution to the NLS equation (\ref{Schy}). We begin with
equation (\ref{grp0}) first. Since 
\begin{equation*}
\nabla \mathring{\psi}=\mathring{\psi}^{\prime }\left( \left\vert \mathbf{y}%
\right\vert \right) \frac{\mathbf{y}}{\left\vert \mathbf{y}\right\vert },
\end{equation*}%
we conclude that (\ref{grp0}) is equivalent to the following \ orthogonality
condition:%
\begin{equation}
\left( \frac{q}{c}\mathbf{A}-\chi \nabla S+m\mathbf{v}\right) \cdot \mathbf{y%
}=\mathbf{0.}  \label{p1o}
\end{equation}%
To treat the above equation we use the concept of a \emph{sphere-tangent
field}. We call a vector field $\mathbf{\breve{V}}\left( \mathbf{y}\right) $
sphere-tangent if it satisfies the orthogonality condition 
\begin{equation}
\mathbf{\breve{V}}\left( \mathbf{y}\right) \cdot \mathbf{y}=0\text{\quad for
all\quad }\mathbf{y.}  \label{ahat0}
\end{equation}%
Obviously, a sphere-tangent vector field is tangent to spheres centered at
the origin. \ Any vector field $\mathbf{V}\left( \mathbf{y}\right) $ \ of
class $C^{1}\left( \mathbb{R}^{3}\right) $ can be uniquely splitted into a
potential field $\nabla P$\ and a sphere-tangent field $\mathbf{\breve{V}}$
which satisfies the orthogonality condition (\ref{ahat0}), namely%
\begin{equation}
\mathbf{V}=\nabla P+\mathbf{\breve{V}}  \label{vsplit0}
\end{equation}%
(see Section \ref{Svsplit} \ for details and explicit formulas). We split
the vector potential $\mathbf{A}$ into its sphere-tangent and gradient
parts, namely 
\begin{equation}
\mathbf{A}\left( t,\mathbf{y}\right) =\mathbf{\breve{A}}\left( t,\mathbf{y}%
\right) +\mathbf{A}_{\nabla }\left( t,\mathbf{y}\right)  \label{Aex}
\end{equation}%
with 
\begin{gather}
\mathbf{\breve{A}}\left( t,\mathbf{y}\right) \cdot \mathbf{y}=0,
\label{Asum} \\
\mathbf{A}_{\nabla }\left( t,\mathbf{y}\right) =\nabla P\left( t,\mathbf{y}%
\right) .  \label{Asum1}
\end{gather}%
Equation (\ref{p1o}) implies that the field $\frac{q}{c}\mathbf{A}-\nabla S+m%
\mathbf{v}$ is purely sphere-tangent, therefore its gradient part is zero
and we conclude that the following equation is equivalent to (\ref{p1o}): 
\begin{equation}
\chi \nabla S=m\mathbf{v}+\frac{q}{c}\mathbf{A}_{\nabla }.  \label{grp1s}
\end{equation}%
Now let us consider equation (\ref{p1}). Expressing $\nabla S$ from (\ref%
{grp1s}) we write (\ref{p1}) in the form 
\begin{equation}
-\chi \partial _{t}S+\mathbf{v}\cdot \left( m\mathbf{v}+\frac{q}{c}\mathbf{A}%
_{\nabla }\right) -q\varphi =\mathrm{i}\frac{\chi }{2m}\frac{q}{c}\nabla
\cdot \mathbf{\breve{A}}+\frac{1}{2m}\left( m\mathbf{v}-\frac{q}{c}\mathbf{%
\breve{A}}\right) ^{2}.  \label{p1a}
\end{equation}%
Taking the imaginary part of (\ref{p1a}), we see that the sphere-tangent
part of the vector potential must satisfy\ the following zero divergency
condition: 
\begin{equation}
\nabla \cdot \mathbf{\breve{A}}=0.  \label{divahat}
\end{equation}%
The real part of (\ref{p1a}) yields equation 
\begin{equation}
\chi \partial _{t}S=m\mathbf{v}^{2}+\frac{q}{c}\mathbf{v}\cdot \mathbf{A}%
_{\nabla }-\frac{1}{2m}\left( m\mathbf{v}-\frac{q}{c}\mathbf{\breve{A}}%
\right) ^{2}-q\varphi .  \label{rep1a}
\end{equation}%
Hence, (\ref{rep1a}) and (\ref{grp1s}) take the form 
\begin{gather}
\chi \partial _{t}S=\frac{1}{2}m\mathbf{v}^{2}+\frac{q}{c}\mathbf{v}\cdot 
\mathbf{A}-\frac{1}{2m}\frac{q^{2}}{c^{2}}\mathbf{\breve{A}}^{2}-q\varphi ,
\label{g0} \\
\chi \nabla S=m\mathbf{v}+\frac{q}{c}\mathbf{A}_{\nabla }.  \label{g1}
\end{gather}%
The above relations give an expression for the 4-gradient of $\chi S$. The
right-hand sides of (\ref{g0}), (\ref{g1}) can be considered as the
coefficients of a differential 1-form. Hence, if the phase $S$ which solves (%
\ref{Schy2}) exists, the form must be exact, and consequently it must be
closed. Conversely, according to Poincare's lemma, the form \ on $\mathbb{R}%
^{4}$ is exact if it is closed \ on $\mathbb{R}^{4}$, and then the phase $S$
can be found by the integration of the differential form. Since the
right-hand side of (\ref{g1}) is the gradient, the form is closed if the
following condition is satisfied: 
\begin{equation}
\nabla \left( \frac{1}{2}m\mathbf{v}^{2}+\frac{q}{c}\mathbf{v}\cdot \mathbf{A%
}-\frac{1}{2m}\frac{q^{2}}{c^{2}}\mathbf{\breve{A}}^{2}-q\varphi \right)
=\partial _{t}\left( m\mathbf{v}+\frac{q}{c}\mathbf{A}_{\nabla }\right) .
\label{compdif}
\end{equation}%
This equation can be interpreted as a balance of forces which allows to
exactly preserve the shape of the wave-corpuscle. \ \emph{Equation (\ref%
{compdif}) together with condition (\ref{divahat}) constitutes the criterion
for the preservation of the shape }$\left\vert \psi \right\vert $\emph{.}

\subsection{Trajectory and phase of a wave-corpuscle}

If equation (\ref{compdif}) holds, the phase function $S$ can be found by
integrating the exact 1-form:%
\begin{equation}
S\left( t,\mathbf{y}\right) =\frac{1}{\chi }\int_{\Gamma }\left( \frac{1}{2}m%
\mathbf{v}^{2}+\frac{q}{c}\mathbf{v}\cdot \mathbf{A}-\frac{1}{2m}\frac{q^{2}%
}{c^{2}}\mathbf{\breve{A}}^{2}-q\varphi \right) \mathrm{d}t+\left( m\mathbf{v%
}+\frac{q}{c}\mathbf{A}_{\nabla }\right) \cdot \mathrm{d}\mathbf{y}
\label{SGam0}
\end{equation}%
where $\Gamma $ is a curve in time-space connecting $\left( 0,\mathbf{0}%
\right) $ with $\left( t,\mathbf{y}\right) $. Since equation (\ref{compdif})
is fulfilled, \ the integral does not depend on the curve. In particular, we
take as $\Gamma =\Gamma _{0}$ a curve formed by two straight-line segments:
the first segment from $\left( 0,\mathbf{0}\right) $ to $\left( t,\mathbf{0}%
\right) $ and the second from $\left( t,\mathbf{0}\right) $ to $\left( t,%
\mathbf{y}\right) $. That yields 
\begin{gather}
S\left( t,\mathbf{y}\right) =\frac{1}{\chi }\int_{\Gamma _{0}}\left( \frac{1%
}{2}m\mathbf{v}^{2}+\frac{q}{c}\mathbf{v}\cdot \mathbf{A}-\frac{1}{2m}\frac{%
q^{2}}{c^{2}}\mathbf{\breve{A}}^{2}-q\varphi \right) \mathrm{d}t+\left( m%
\mathbf{v}+\frac{q}{c}\mathbf{A}_{\nabla }\right) \cdot \mathrm{d}\mathbf{y}
\label{SGam} \\
=m\mathbf{v}\cdot \mathbf{y}+\frac{q}{c}\mathbf{y}\cdot \mathbf{A}\left( t,%
\mathbf{0}\right) +s_{\mathrm{p}}\left( t\right) +s_{\mathrm{p2}}\left( t,%
\mathbf{y}\right) ,  \notag
\end{gather}%
where%
\begin{gather}
s_{\mathrm{p}}\left( t\right) =\frac{1}{\chi }\int_{0}^{t}\left( \frac{1}{2}m%
\mathbf{v}^{2}\left( t\right) +\frac{q}{\mathrm{c}}\mathbf{v}\cdot \mathbf{A}%
\left( t,0\right) -q\varphi \left( t,0\right) \right) \mathrm{d}t,
\label{SGam1} \\
s_{\mathrm{p2}}\left( t,\mathbf{y}\right) =\frac{q}{\chi \mathrm{c}}%
\int_{0}^{1}\mathbf{y}\cdot \left( \mathbf{A}\left( t,s\mathbf{y}\right) -%
\mathbf{A}\left( t,\mathbf{0}\right) \right) \mathrm{d}s,  \label{SGam2}
\end{gather}%
and $s_{\mathrm{p2}}\left( t,\mathbf{y}\right) $\ is at least quadratic in $%
\mathbf{y}$. In the above derivation we used that $\mathbf{\breve{A}}\left(
t,\mathbf{0}\right) =0$, $\mathbf{y}\cdot \mathbf{\breve{A}}\left( t,\mathbf{%
y}\right) =0$, and$\ \ \mathbf{A}_{\nabla }\left( t,\mathbf{0}\right) =%
\mathbf{A}\left( t,\mathbf{0}\right) $.

Singling out the linear part of the phase we can write above formulas in the
form 
\begin{equation}
S\left( t,\mathbf{y}\right) =\frac{1}{\chi }m\mathbf{\tilde{v}}\cdot \mathbf{%
y}+s_{\mathrm{p}}\left( t\right) +s_{\mathrm{p2}}\left( t,\mathbf{y}\right) ,
\label{Sgamv}
\end{equation}%
where%
\begin{equation}
\mathbf{\tilde{v}=v}+\frac{q}{mc}\mathbf{A}\left( t,\mathbf{0}\right) ,
\label{vhat0}
\end{equation}%
\begin{equation}
s_{\mathrm{p}}\left( t\right) =\frac{1}{\chi }\int_{0}^{t}\left( \frac{1}{2}m%
\mathbf{\tilde{v}}^{2}\left( t\right) -\frac{q^{2}}{2mc^{2}}\left( \mathbf{A}%
\left( t,0\right) \right) ^{2}-q\varphi \left( t,0\right) \right) \mathrm{d}%
t.  \label{sp1}
\end{equation}%
The integrability condition (\ref{compdif}) can be written now in the form 
\begin{equation}
m\partial _{t}^{2}\mathbf{r}=\frac{q}{c}\nabla \left( \mathbf{v}\cdot 
\mathbf{A}\right) -\frac{q^{2}}{c^{2}}\frac{1}{2m}\nabla \mathbf{\breve{A}}%
^{2}-q\nabla \varphi -\frac{q}{c}\partial _{t}\mathbf{A}_{\nabla }.
\label{intcon1}
\end{equation}%
The above equation together with (\ref{divahat}) is a system of equations
which involves the EM potentials $\varphi ,\mathbf{A}$ and the corpuscle
center trajectory $\mathbf{r}\left( t\right) $. Its fulfillment guarantees
that the wave-corpuscle preserves its shape in the dynamics described by the
NLS equation. We refer to equations (\ref{intcon1}), (\ref{divahat}) \emph{%
non-relativistic wave-corpuscle dynamic balance conditions}. \ Obviously the
conditions do not depend on the nonlinearity $G$.

By setting \ $\mathbf{y}=0$ in (\ref{intcon1}) and taking into account that $%
\mathbf{\breve{A}}_{\mathrm{ex}}\left( t,0\right) =0$ we obtain the \emph{%
point balance condition:}%
\begin{equation}
m\partial _{t}^{2}\mathbf{r}=\frac{q}{c}\nabla \left( \mathbf{v}\cdot 
\mathbf{A}\right) \left( t,0\right) -q\nabla \varphi \left( t,0\right) -%
\frac{q}{c}\partial _{t}\mathbf{A}\left( t,0\right) .  \label{eq0}
\end{equation}%
To interpret this condition, \ we note that $\mathbf{y}=0$ in the moving
frame in the above equation \ corresponds to $\mathbf{x}=\mathbf{r}$ \ in
the resting frame\ and 
\begin{equation}
\partial _{t}\mathbf{A}\left( t,\mathbf{x}\right) _{\mathbf{x=r}}=\partial
_{t}\mathbf{A}\left( t,\mathbf{y}\right) _{\mathbf{y=0}}-\mathbf{v}\cdot
\nabla \mathbf{A}\left( t,\mathbf{y}\right) _{\mathbf{y=0}},  \label{dtmov}
\end{equation}%
hence the point balance condition (\ref{eq0}) takes in the resting frame the
form \ 
\begin{equation}
m\partial _{t}^{2}\mathbf{r}=\frac{q}{c}\nabla \left( \mathbf{v}\cdot 
\mathbf{A}\right) \left( t,\mathbf{r}\right) -q\nabla \varphi \left( t,%
\mathbf{r}\right) -\frac{q}{c}\partial _{t}\mathbf{A}\left( t,\mathbf{r}%
\right) -\frac{q}{c}\mathbf{v}\cdot \nabla \mathbf{A}\left( t,\mathbf{r}%
\right) .  \label{eq01}
\end{equation}%
Recall that the expression for the Lorentz force in terms of the EM
potentials is given by the formula 
\begin{equation*}
\mathbf{f}_{\mathrm{Lor}}=-q\nabla \varphi -\frac{q}{c}\partial _{t}\mathbf{%
\mathbf{A}}+\frac{q}{c}\nabla \left( \mathbf{v}\cdot \mathbf{\mathbf{A}}%
\right) -\frac{q}{c}\left( \mathbf{v}\cdot \nabla \right) \mathbf{\mathbf{A,}%
}
\end{equation*}%
therefore the right-hand side of (\ref{eq01}) coincides with the Lorentz
force (\ref{fLor0}). Hence, \ the point balance condition (\ref{eq0}) in $%
\mathbf{x}$-coordinates has the form 
\begin{gather}
m\partial _{t}^{2}\mathbf{r}=\mathbf{f}_{\mathrm{Lor}}\left( t,\mathbf{r}%
\right) \mathbf{,}  \label{flor} \\
\mathbf{f}_{\mathrm{Lor}}\left( t,\mathbf{r}\right) =q\mathbf{E}\left( t,%
\mathbf{r}\right) +\frac{q}{c}\mathbf{v}\times \mathbf{B}\left( t,\mathbf{r}%
\right)  \label{flor1}
\end{gather}%
where EM fields $\mathbf{E}$ and $\mathbf{B}$ are given by (\ref{maxw3a}). 
\emph{Hence the point balance condition coincides with Newton's law of
motion (\ref{Newt0}) of a charged point \ subjected to the Lorentz force }$%
\mathbf{f}_{\mathrm{Lor}}$.

We always assume that the trajectory $\mathbf{r}\left( t\right) $ satisfies
the point balance condition (\ref{flor}), therefore the integrability
condition (\ref{intcon1}) after taking into account \ (\ref{eq0}) \ can be
written in the form%
\begin{gather}
\frac{q}{c}\left( \nabla \left( \mathbf{v}\cdot \mathbf{A}\left( t,\mathbf{y}%
\right) -\nabla \left( \mathbf{v}\cdot \mathbf{A}\right) \left( t,0\right)
\right) \right) -\frac{q^{2}}{c^{2}}\frac{1}{2m}\nabla \mathbf{\breve{A}}%
^{2}-q\left( \nabla \varphi \left( t,\mathbf{y}\right) -\nabla \varphi
\left( t,0\right) \right)  \label{intcon2} \\
-\frac{q}{c}\partial _{t}\left( \mathbf{A}_{\nabla }\left( t,\mathbf{y}%
\right) -\mathbf{A}_{\nabla }\left( t,\mathbf{0}\right) \right) =\mathbf{0} 
\notag
\end{gather}%
where the left hand side \ explicitly vanishes for spatially constant fields 
$\mathbf{A}$ and constant $\nabla \varphi $.

Now let us discuss the universality of the non-relativistic dynamic balance
conditions. \ The conditions (\ref{divahat}), (\ref{intcon2}) of
preservation of the shape $\left\vert \psi \right\vert $ are \emph{universal}%
, namely they do not depend on the nonlinearity $G\ $or on the form factor $%
\mathring{\psi}$. Recall that when deriving the conditions we set
coefficients at $\nabla \mathring{\psi}$ and $\mathring{\psi}$ in (\ref%
{Schy2}) to be zero. We would like to show that this is a necessary
requirement\ \ for the conditions to be universal.

Equation (\ref{Schy2}) has the form 
\begin{equation*}
Q_{0}\mathring{\psi}\left( r\right) +\mathring{\psi}^{\prime }\left(
r\right) Q_{1}\cdot \mathbf{y/}r=0,
\end{equation*}%
where $\mathring{\psi}\left( r\right) $ is a form factor, and coefficients $%
Q_{0},Q_{1}$ obviously do not depend on $\mathring{\psi}$. If $Q_{1}\cdot 
\mathbf{y}\ \ $is not identically zero, \ we can separate the variables: \ 
\begin{equation*}
\frac{Q_{0}}{rQ_{1}\cdot \mathbf{y}}=-\partial _{r}\ln \mathring{\psi}\left(
r\right) ,
\end{equation*}%
where the left-hand side is $\mathring{\psi}$-independent and the right-hand
side depends on $\mathring{\psi}$. It is impossible, and hence $Q_{1}\cdot 
\mathbf{y}$ must be zero as well $Q_{0}$, leading to equations (\ref{grp0}),
(\ref{p1}).

\subsection{Wave-corpuscles in the EM field\label{Thexactq}}

The simplest example of fields $\varphi ,\mathbf{A}$ for which dynamic
balance conditions are fulfilled are spatially constant $\varphi \left(
t\right) ,\mathbf{A}\left( t\right) $; equations (\ref{divahat}), (\ref%
{intcon2}) are obviously fulfilled. \ Now we construct a more general
example of the fulfillment of the condition (\ref{intcon2}). In this example
we prescribe arbitrary EM potentials which are linear in $\mathbf{x}$, take
a trajectory which satisfies Newton's law (\ref{flor}), and assume that
second and higher order components of the EM potentials expansion at $%
\mathbf{r}\left( t\right) $ must satisfy certain restrictions. Namely, we
assume that $\mathbf{A}$ involves a linear in $\mathbf{y}=\mathbf{x}-\mathbf{%
r}$ part which can be given by an arbitrary $3\times 3$ matrix $\mathbf{A}%
_{1}\left( t\right) $ with time-dependent elements: 
\begin{equation}
\mathbf{A}_{1}\left( t,\mathbf{y}\right) =\mathbf{A}_{1}\left( t\right) 
\mathbf{y}=\mathbf{y}\cdot \nabla \mathbf{A}\left( t,\mathbf{0}\right) .
\label{aexarb}
\end{equation}%
We assume that quadratic and higher order components of the sphere-tangent
part of $\mathbf{A}$ are set to zero, but allow an arbitrary potential part 
\begin{equation}
\mathbf{A}\left( t,\mathbf{y}\right) =\mathbf{A}\left( t,\mathbf{0}\right) +%
\mathbf{A}_{1}\left( t,\mathbf{y}\right) +\mathbf{A}_{\nabla 2}\left( t,%
\mathbf{y}\right) ,  \label{aexlin}
\end{equation}%
where $\mathbf{A}_{\nabla 2}$ has at least the second order zero at the
origin and is potential:%
\begin{equation*}
\mathbf{A}_{\nabla 2}\left( t,\mathbf{y}\right) =\nabla P_{3}
\end{equation*}%
with an arbitrary $P_{3}$. For such a field, its potential and
sphere-tangent parts are given respectively by formulas 
\begin{gather}
\mathbf{A}_{\nabla }\left( t,\mathbf{y}\right) =\nabla P_{1}+\nabla
P_{2}+\nabla P_{3},  \label{aex1} \\
\mathbf{\breve{A}}\left( t,\mathbf{y}\right) =\frac{1}{2}\mathbf{A}%
_{1}\left( t\right) \mathbf{y}-\frac{1}{2}\mathbf{A}_{1}^{\mathrm{T}}\left(
t\right) \mathbf{y},  \label{aex1a}
\end{gather}%
where%
\begin{equation}
P_{1}=\mathbf{y}\cdot \mathbf{A}\left( t,\mathbf{0}\right) ,\qquad
P_{2}\left( t,\mathbf{y}\right) =\frac{1}{2}\left( \mathbf{y\cdot A}%
_{1}\left( t\right) \mathbf{y}\right) ,  \label{p1p2}
\end{equation}%
$\mathbf{A}_{1}^{\mathrm{T}}$ stands for $\mathbf{A}_{1}$\ transposed, and $%
P_{3}$ is a function with zero of at least the third degree at the origin.
Note that an action of an anti-symmetric matrix can be written using the
cross product, therefore (\ref{aex1a}) can be written as follows: 
\begin{equation}
\mathbf{\breve{A}}\left( t,\mathbf{y}\right) =\frac{1}{2}\mathbf{\breve{B}}%
\left( t\right) \times \mathbf{y,}  \label{aex1b}
\end{equation}%
where $\mathbf{\breve{B}}=\nabla \times \mathbf{\breve{A}}$. The electric
potential $\ \varphi $ has the form \ 
\begin{equation}
\varphi =\varphi \left( t,\mathbf{0}\right) +\mathbf{y}\cdot \nabla \varphi
\left( t,\mathbf{0}\right) +\varphi _{2}\left( t,\mathbf{y}\right) ,
\label{fiex2}
\end{equation}%
where $\varphi \left( t,\mathbf{0}\right) $ and $\nabla \varphi \left( t,%
\mathbf{0}\right) $ are given continuously differentiable functions of $t,$
and $\varphi _{2}\left( t,\mathbf{y}\right) $ \ has order two or higher in $%
\mathbf{y}$ \ and is subject to the condition (\ref{fiex2a}) we formulate
below.

Let us verify the conditions which guarantee that the wave-corpuscle is an
exact solution. Notice that condition (\ref{divahat}) for $\mathbf{\breve{A}}
$ defined by (\ref{aex1a}) is fulfilled. Hence, it is sufficient to satisfy (%
\ref{intcon2}), which takes the form%
\begin{equation}
\frac{q}{c}\nabla \left( \mathbf{v}\cdot \nabla P_{3}\right) -\frac{q^{2}}{%
c^{2}}\frac{1}{2m}\nabla \mathbf{\breve{A}}^{2}-q\nabla \varphi _{\mathrm{2}%
}\left( t,\mathbf{y}\right) -\frac{q}{c}\partial _{t}\left( \nabla
P_{2}\left( t,\mathbf{y}\right) +\nabla P_{3}\left( t,\mathbf{y}\right)
\right) =\mathbf{0.}  \label{grp}
\end{equation}%
To satisfy (\ref{grp}), we set 
\begin{equation}
\varphi _{2}\left( t,\mathbf{y}\right) =-\frac{q}{c^{2}}\frac{1}{2m}\mathbf{%
\breve{A}}^{2}-\frac{1}{c}\partial _{t}\left( P_{2}\left( t,\mathbf{y}%
\right) +P_{3}\left( t,\mathbf{y}\right) \right) +\frac{1}{c}\mathbf{v}\cdot
\nabla P_{3}.  \label{fiex2a}
\end{equation}%
We want now to determine the phase $S$ of the wave-corpuscle using (\ref%
{SGam1}), (\ref{SGam2}). If $P_{3}$ is homogenious of third degree, we
obtain 
\begin{equation}
s_{\mathrm{p}}\left( t\right) =\frac{1}{\chi }\int_{0}^{t}\left( \frac{1}{2}m%
\mathbf{v}^{2}\left( t\right) +\frac{q}{c}\mathbf{v}\cdot \mathbf{A}\left(
t,0\right) -q\varphi \left( t,0\right) \right) \mathrm{d}t,  \label{sp2}
\end{equation}%
\begin{gather}
s_{\mathrm{p2}}\left( t,\mathbf{y}\right) =\frac{q}{c}\frac{1}{\chi }%
\int_{0}^{1}\mathbf{y}\cdot \mathbf{\mathbf{A}}_{1}\left( t,\mathbf{y}%
\right) sds+\frac{q}{c}\frac{1}{\chi }\int_{0}^{1}\mathbf{y}\cdot \mathbf{A}%
_{\nabla 2}\left( t,\mathbf{y}\right) s^{2}\mathrm{d}s  \label{sp2a} \\
=\frac{q}{2c}\frac{1}{\chi }\mathbf{y}\cdot \mathbf{\mathbf{A}}_{1}\left( t,%
\mathbf{y}\right) +\frac{q}{3c}\frac{1}{\chi }\mathbf{y}\cdot \mathbf{A}%
_{\nabla 2}\left( t,\mathbf{y}\right) .  \notag
\end{gather}%
In this case the phase function of the wave-corpuscle involves a term $s_{%
\mathrm{p2}}\left( t,\mathbf{y}\right) $.

Formula (\ref{psil00}) takes the form 
\begin{equation}
\psi \left( t,\mathbf{x}\right) =\mathrm{e}^{\mathrm{i}S}\mathring{\psi}%
\left( \left\vert \mathbf{x}-\mathbf{r}\left( t\right) \right\vert \right) ,
\label{psil0a}
\end{equation}%
\begin{equation}
S\left( t,\mathbf{y}\right) =m\frac{1}{\chi }\mathbf{v}\cdot \mathbf{y}+%
\frac{q}{c}\frac{1}{\chi }\mathbf{y}\cdot \mathbf{A}\left( t,\mathbf{0}%
\right) +s_{\mathrm{p}}\left( t\right) +s_{\mathrm{p2}}\left( t,\mathbf{y}%
\right) .  \label{Sdef}
\end{equation}%
where $\;\mathbf{y}=\mathbf{x}-\mathbf{r}\left( t\right) ,\mathbf{v}%
=\partial _{t}\mathbf{r}$.

If \ the external magnetic field satisfies \ the following anti-symmetry
conditions 
\begin{gather}
\mathbf{y}\cdot \mathbf{\mathbf{A}}_{1}\left( t,\mathbf{y}\right) =0,
\label{antmag} \\
\mathbf{y}\cdot \mathbf{A}_{\nabla 2}\left( t,\mathbf{y}\right) =0  \notag
\end{gather}%
then $s_{\mathrm{p2}}=0$, and the phase function is linear in $\mathbf{y}$.
\ 

The above calculations can be summarized in the following statement

\begin{theorem}
\label{Twcem} \ Let the potentials $\varphi ,\mathbf{A}$ have the form (\ref%
{fiex2}), (\ref{aexarb}), (\ref{aexlin}) where $\varphi \left( t,\mathbf{0}%
\right) $ and $\nabla \varphi \left( t,\mathbf{0}\right) $ are given
continuously differentiable functions of $t$. Suppose also that $\mathbf{A}%
\left( t,\mathbf{0}\right) $ is a given continuously differentiable function
of $t$, $\mathbf{A}_{1}\left( t\right) $ is an arbitrary $3\times 3$ matrix
which is continuously differentiably depends on $t$, $\nabla P_{3}$ is a
continuously differentiable function of $t$ and $\mathbf{y}$. Let the
quadratic part $\varphi _{2}$ of the potential $\varphi $ satisfy (\ref%
{fiex2a}). Let also trajectory $\mathbf{r}\left( t\right) $ satisfies
Newton's equation (\ref{flor})-(\ref{flor1}). Then the wave-corpuscle
defined by formula (\ref{psil0a})-(\ref{Sdef}) is a solution to NLS equation
(\ref{NLS}).
\end{theorem}

\begin{remark}
\label{R:unia}The above construction does not depend on the nonlinearity $%
G^{\prime }=G_{a}^{\prime }$ as long as (\ref{nop40}) is satisfied. It is
also uniform with respect to $a>0$, and the dependence on $a$ in (\ref%
{psil0a}) is only through $\mathring{\psi}\left( |\mathbf{x}-\mathbf{r}%
|\right) =a^{-3/2}\mathring{\psi}_{1}\left( a^{-1}|\mathbf{x}-\mathbf{r}%
|\right) $. Obviously, if $\psi \left( t,\mathbf{x}\right) $ is defined by (%
\ref{psil00}) then $|\psi \left( t,\mathbf{x}\right) |^{2}\rightarrow \delta
\left( \mathbf{x}-\mathbf{r}\right) $ as $a\rightarrow 0$.
\end{remark}

\begin{remark}
The form (\ref{psil00}) of exact solutions is the same as the WKB ansatz in
the quasi-classical approach, \cite{MaslovFedorjuk}. The trajectories of the
charge center coincide with trajectories that can be found by applying
well-known quasiclassical asymptotics to solutions of (\ref{NLS}) if one
neglects the nonlinearity. Note though that there are two important effects
of the nonlinearity not presented in the standard quasiclassical approach.
First of all, due to the nonlinearity the charge preserves its shape in the
course of evolution on unbounded time intervals whereas in the linear model
any wavepacket disperses over time. Second of all, the quasiclassical
asymptotic expansions produce infinite asymptotic series which provide for a
formal solution, whereas the properly introduced nonlinearity as in (\ref%
{nop40}), (\ref{intpsa}) allows one to obtain an exact solution. For a
treatment of mathematical aspects of the approach to nonlinear wave
mechanics based on the WKB asymptotic expansions we refer the reader to \cite%
{MaslovFedorjuk}, \cite{Komech05} and references therein.
\end{remark}

\begin{remark}
We can use in the definition of a wave-corpuscle (\ref{psil00})\ a radial
form factor $\mathring{\psi}$\ which instead of (\ref{nop40}) satisfies the
eigenvalue problem 
\begin{equation}
-\nabla ^{2}\mathring{\psi}+G^{\prime }(\left\vert \mathring{\psi}%
\right\vert ^{2})\mathring{\psi}=\lambda \mathring{\psi}.  \label{eipr}
\end{equation}%
Obviously the above equation (\ref{eipr}) can be considered as the steady
state equation (\ref{nop40}) \ with a modified nonlinearity $G^{\prime
}-\lambda $. Note that $\psi $ is a solution of the NLS equation (\ref{NLS})
with the nonlinearity $G^{\prime }$ if and only if the function $\mathrm{e}%
^{-\mathrm{i}\frac{\chi }{2m}\lambda t}\psi $ \ is a solution of the NLS
equation with the nonlinearity $G^{\prime }-\lambda $. Therefore the
wave-corpuscle solutions constructed in Theorem \ref{Twcem} based on $%
\mathring{\psi}$ with the nonlinearity $G^{\prime }-\lambda $ provide
solutions to the NLS equation with the original nonlinearity, \ one has only
to add to the phase function $S$ an additional term $\frac{\chi }{2m}\lambda
t$. \ Existence of many solutions of the nonlinear eigenvalue problems of
the form (\ref{eipr}) \ was proved in many papers, see \cite%
{BerestyckiLions83I}, \cite{BerestyckiLions83II}, \cite{Bialynicki1}, \cite%
{Heid} . Therefore there are many wave-corpuscle solutions of a given NLS\
equation with the same trajectory of motion $\mathbf{r}\left( t\right) $ \
and the same EM fields which correspond to different form factors.
\end{remark}

\subsection{Wave-corpuscles as concentrating solutions\label{Swcconc}}

The wave-corpuscles constructed in Theorem \ref{Twcem} provide an example of
concentrating solutions. In order to see that we have to verify all the
requirements of Definition \ref{Dlocconverges}. It is important to note that
the dependence on $a$ in (\ref{psil00}) is only through $\mathring{\psi}%
\left( |\mathbf{x}-\mathbf{r}|\right) =a^{-3/2}\mathring{\psi}_{1}\left(
a^{-1}|\mathbf{x}-\mathbf{r}|\right) $ \ where $\mathring{\psi}_{1}$ is a
given smooth function. The verification of conditions imposed in the
Definition \ref{Dlocconverges} is straightforward. For instance, 
\begin{gather}
\bar{\rho}_{n}=q\int_{\Omega \left( \mathbf{\hat{r}}(t),R_{n}\right)
}a^{-3}\left\vert \mathring{\psi}_{1}\left( a^{-1}|\mathbf{x}-\mathbf{r}%
|\right) \right\vert ^{2}\,\mathrm{d}\mathbf{x}  \label{limrho} \\
=q\int_{\Omega \left( \mathbf{0},R_{n}/a\right) }\mathring{\psi}%
_{1}^{2}\left( |\mathbf{y}|\right) \,\mathrm{d}\mathbf{y}\rightarrow q\int_{%
\mathbb{R}^{3}}\mathring{\psi}_{1}^{2}\left( |\mathbf{y}|\right) \,\mathrm{d}%
\mathbf{y,}  \notag
\end{gather}%
where $R_{n}/a\rightarrow \infty $ \ according to (\ref{theninfs}) and using
that we obtain (\ref{Egrcs}) and (\ref{rconv}). To estimate integral (\ref%
{locpsibounds}) we note that according to the definition of momentum density
(\ref{schr5}), (\ref{covNLS})%
\begin{equation}
\mathbf{P}=\mathrm{i}\chi \frac{1}{2}\left[ \frac{\nabla \psi ^{\ast }}{\psi
^{\ast }}-\frac{\nabla \psi }{\psi }+2\frac{\mathrm{i}q}{\chi \mathrm{c}}%
\mathbf{A}\right] \psi ^{\ast }\psi ,  \label{PNLS}
\end{equation}%
and for the wave-corpuscles 
\begin{equation}
\mathbf{P}=\chi \left[ \nabla S-\frac{q}{\chi \mathrm{c}}\mathbf{A}\right] 
\mathring{\psi}^{2}.  \label{Pim}
\end{equation}%
Since the phase $S$ in (\ref{Sdef}) is a smooth function which does not
depend on $a$, and $\mathbf{A}$\ satisfies (\ref{Abounds}), the estimate (%
\ref{locpsibounds}) \ can be obtained using (\ref{limrho}). Using (\ref{Pim}%
) and (\ref{limrho}) we obtain (\ref{Q01Pto0}):%
\begin{equation}
\left\vert Q_{01}\right\vert \leq \left\vert t-t_{0}\right\vert
\max_{T_{-}\leq s\leq T_{+}}\left\vert \partial _{t}\mathbf{\hat{r}}%
\right\vert \max_{T_{-}\leq s\leq T_{+}}\int_{\partial \Omega
_{n}}\left\vert \mathbf{P}\right\vert \,\mathrm{d}\sigma \rightarrow 0.
\label{Q01to0}
\end{equation}%
Similarly we obtain (\ref{Q20to0}), (\ref{Q22to0}) and\ (\ref{Q23to0}). Note
that $\left\vert \mathbf{E}_{n}\left( 0\right) -\mathbf{E}_{\infty }\left(
0\right) \right\vert =\left\vert \mathbf{E}\left( 0\right) -\mathbf{E}%
_{\infty }\left( 0\right) \right\vert =0$ and $\left\vert \mathbf{B}%
_{n}\left( 0\right) -\mathbf{B}_{\infty }\left( 0\right) \right\vert
=\left\vert \mathbf{B}\left( 0\right) -\mathbf{B}_{\infty }\left( 0\right)
\right\vert =0$ according to (\ref{ebinf0a}). Using continuity of $\mathbf{E}%
\left( t,\mathbf{y}\right) $ and $\mathbf{B}\left( t,\mathbf{y}\right) $ we
conclude that $\left\vert \mathbf{E}_{n}-\mathbf{E}_{\infty }\ \right\vert
\rightarrow 0$ and $\left\vert \mathbf{B}_{n}-\mathbf{B}\right\vert
\rightarrow 0$ in $\Omega _{n}$ since $R_{n}\rightarrow 0$. Therefore
filfillment of (\ref{Q3to0}) follows from (\ref{limrho}) and (\ref{Pim}).

To obtain (\ref{Q0Tto0}) we split the tensor $T^{ij}$\ into diagonal and
non-diagonal parts given by (\ref{emfr10}) and (\ref{Tij}): 
\begin{equation}
T^{ij}=T_{i\neq j}^{ij}+T_{i=j}^{ij}.  \label{Tijsp}
\end{equation}%
According to (\ref{QP0}) $\ $%
\begin{equation*}
Q_{0}=Q_{0,\neq }+Q_{0,=}
\end{equation*}%
where%
\begin{equation*}
Q_{0,\neq }=\int_{t_{0}}^{t}\int_{\partial \Omega _{n}}\mathbf{\bar{n}}%
_{i}T_{i\neq j}^{ij}\,\mathrm{d}\sigma \mathrm{d}t^{\prime },\qquad
Q_{0,=}=\int_{t_{0}}^{t}\int_{\partial \Omega _{n}}\mathbf{\bar{n}}%
_{i}T_{i=j}^{ij}\,\mathrm{d}\sigma \mathrm{d}t^{\prime }.
\end{equation*}%
According to (\ref{Tij}) \ to estimate $\left\vert Q_{0,\neq }\right\vert \ $
it is sufficient to prove that 
\begin{equation}
\max_{T_{-}\leq t\leq T_{+}}\int_{\partial \Omega _{n}}(\frac{\chi ^{2}}{2m}%
\left\vert \nabla \psi \left( t,\mathbf{x}\right) \right\vert
^{2}+\left\vert \psi \left( t,\mathbf{x}\right) \right\vert ^{2})\,\mathrm{d}%
\sigma \rightarrow 0.  \label{psioutas}
\end{equation}%
Using this estimate and \ (\ref{Tij}) we obtain%
\begin{equation}
\left\vert Q_{0,\neq }\right\vert \leq \left\vert t-t_{0}\right\vert
\max_{T_{-}\leq s\leq T_{+}}\int_{\partial \Omega _{n}}\left\vert T_{i\neq
j}^{ij}\right\vert \,\mathrm{d}\sigma \rightarrow 0.  \label{Q0to0}
\end{equation}%
To estimate \ $Q_{0,=}$\ we use \ (\ref{emfr10}) 
\begin{gather}
\int_{\partial \Omega _{n}}\mathbf{\bar{n}}_{i}T^{ii}\mathrm{d}\sigma
=-\int_{\partial \Omega _{n}}\frac{\chi ^{2}}{m}\mathbf{\bar{n}}_{i}\tilde{%
\partial}_{i}\psi \tilde{\partial}_{i}^{\ast }\psi ^{\ast }\mathrm{d}\sigma
\label{nTii} \\
+\int_{\partial \Omega _{n}}\left[ \frac{\chi ^{2}}{2m}\left( G\left(
\left\vert \psi \right\vert ^{2}\right) +\left\vert \tilde{\nabla}\psi
\right\vert ^{2}\right) +\mathrm{i}\frac{\chi }{2}\left( \psi \tilde{\partial%
}_{t}^{\ast }\psi ^{\ast }-\psi ^{\ast }\tilde{\partial}_{t}\psi \right) %
\right] \sum_{i}\mathbf{\bar{n}}_{i}\mathrm{d}\sigma  \notag
\end{gather}%
Since $\ \left\vert \psi \right\vert ^{2}=\left\vert \mathring{\psi}%
\right\vert ^{2}$ is a radial function and $\mathbf{\bar{n}}%
_{i}=y_{i}/\left\vert \mathbf{y}\right\vert $ \ is odd with respect to $%
y_{i} $--reflection, we see that in the above integral 
\begin{equation}
\int_{\partial \Omega _{n}}G\left( \left\vert \psi \right\vert ^{2}\right) 
\mathbf{\bar{n}}_{i}\mathrm{d}\sigma =0.  \label{ig0}
\end{equation}%
We also note that 
\begin{gather}
\mathrm{i}\frac{\chi }{2}\left( \psi \tilde{\partial}_{t}^{\ast }\psi ^{\ast
}-\psi ^{\ast }\tilde{\partial}_{t}\psi \right) =\mathrm{i}\frac{\chi }{2}%
\left( \psi \partial _{t}\psi ^{\ast }-\psi ^{\ast }\partial _{t}\psi
\right) +q\varphi \left\vert \psi \right\vert ^{2}  \label{psdtps} \\
=\chi \left\vert \psi \right\vert ^{2}\partial _{t}S+q\varphi \left\vert
\psi \right\vert ^{2}  \notag
\end{gather}%
where $\partial _{t}S$ and $\varphi $ are bounded functions \ in $\Omega
_{n} $ . Under the assumption (\ref{psioutas}) straightforward estimates
produce that%
\begin{equation*}
\mathrm{i}\frac{\chi }{2}\int_{\partial \Omega _{n}}\left( \psi \tilde{%
\partial}_{t}^{\ast }\psi ^{\ast }-\psi ^{\ast }\tilde{\partial}_{t}\psi
\right) \mathrm{d}\sigma \rightarrow 0,
\end{equation*}%
\begin{equation*}
\int_{\partial \Omega _{n}}\frac{\chi ^{2}}{m}\mathbf{\bar{n}}_{i}\tilde{%
\partial}_{i}\psi \tilde{\partial}_{i}^{\ast }\psi ^{\ast }\mathrm{d}\sigma
\rightarrow 0,
\end{equation*}%
\begin{equation*}
\int_{\partial \Omega _{n}}\frac{\chi ^{2}}{2m}\left\vert \tilde{\nabla}\psi
\right\vert ^{2}\mathrm{d}\sigma \rightarrow 0
\end{equation*}%
uniformly on $\left[ T_{-},T_{+}\right] $. Hence, if (\ref{psioutas}) holds,
we obtain that 
\begin{equation}
Q_{0,=}\rightarrow 0,  \label{Q0to01}
\end{equation}%
and taking into account (\ref{Q0to0}) we conclude that (\ref{Q0Tto0}) holds.

Now let us prove that (\ref{psioutas}) holds under certain decay conditions.
\ Since the phase $S$ in (\ref{Sdef}) is bounded and have bounded
derivatives in $\Omega \left( \mathbf{\hat{r}}(t),R_{n}\right) $, to obtain (%
\ref{psioutas})\ it is sufficient to estimate surface integrals of $%
\left\vert \nabla \mathring{\psi}\right\vert ^{2}\ \ $and $\left\vert 
\mathring{\psi}\right\vert ^{2}$. Obviously, \ 

\begin{gather}
\int_{\partial \Omega \left( \mathbf{\hat{r}}(t),R_{n}\right)
}a^{-3}\left\vert \nabla \mathring{\psi}_{1}\left( a^{-1}|\mathbf{x}-\mathbf{%
r}|\right) \right\vert ^{2}\,\mathrm{d}\sigma  \label{intdom} \\
=4\pi R_{n}^{2}a^{-5}\left\vert \mathring{\psi}_{1}^{\prime }\left(
R_{n}/a\right) \right\vert ^{2}=4\pi aR_{n}^{-4}\theta ^{6}\left\vert 
\mathring{\psi}_{1}^{\prime }\left( \theta \right) \right\vert ^{2},  \notag
\end{gather}%
\begin{gather}
\int_{\partial \Omega \left( \mathbf{\hat{r}}(t),R_{n}\right)
}a^{-3}\left\vert \mathring{\psi}_{1}\left( a^{-1}|\mathbf{x}-\mathbf{r}%
|\right) \right\vert ^{2}\,\mathrm{d}\sigma =  \label{intdom0} \\
=4\pi a^{-3}R_{n}^{2}\left\vert \mathring{\psi}_{1}\left( R_{n}/a\right)
\right\vert ^{2}=4\pi aR_{n}^{-2}\theta ^{4}\left\vert \mathring{\psi}%
_{1}\left( \theta \right) \right\vert ^{2},  \notag
\end{gather}%
where $R_{n}/a=\theta \rightarrow \infty $, $R_{n}\rightarrow 0$. \ We
assume that the form factor $\ \mathring{\psi}_{1}\left( r\right) $ and its
derivative $\ \mathring{\psi}_{1}^{\prime }\left( r\right) $ \ satisfy the
following decay conditions:%
\begin{equation}
\theta ^{2}\left\vert \mathring{\psi}_{1}\left( \theta \right) \right\vert
\leq C\text{ \ as \ }\theta \rightarrow \infty ,  \label{psi1}
\end{equation}%
\begin{equation}
\theta ^{3}\left\vert \mathring{\psi}_{1}^{\prime }\left( \theta \right)
\right\vert \leq C\text{ \ as \ }\theta \rightarrow \infty ,  \label{rpsi1}
\end{equation}%
and we take such$\ \ a_{n},R_{n}\ $ that%
\begin{equation}
a_{n}R_{n}^{-4}\rightarrow 0,\qquad a_{n}\rightarrow 0,\qquad
R_{n}\rightarrow 0.  \label{ar4}
\end{equation}%
Under these assumptions we conclude that (\ref{psioutas}) is fulfilled.

Summing up the above arguments we obtain the following statement:

\begin{theorem}
\label{Twcconc}Let $\mathring{\psi}_{1}$ satisfy (\ref{psi1}), (\ref{rpsi1})
and $a_{n},R_{n}$ \ satisfy (\ref{ar4}). Let $\psi =\psi _{n}$ \ be
wave-corpuscle solutions of the NLS equation constructed in Theorem \ref%
{Twcem}. Then the sequence $\psi _{n}$ concentrates at the trajectory $%
\mathbf{\hat{r}}(t)=\mathbf{r}(t)$.
\end{theorem}

\section{Concentration of asymptotic solutions\label{Sconcas}}

From the proofs of Section \ref{SNLSpoint} one can see that in the
derivation of the Newtonian dynamics we use only the conservation laws (\ref%
{schr9}) and (\ref{momNLS}). Since we use only asymptotic properties, it is
natural to consider fields which satisfy NLS equations and the conservation
laws not exactly, but approximately. Namely, we assume now that the
conservation laws (\ref{schr9}) and (\ref{momNLS}) are replaced by 
\begin{equation}
\partial _{t}\rho +\partial _{t}\rho ^{\prime }+\nabla \cdot \mathbf{J}%
+\nabla \cdot \mathbf{J}^{\prime }=0,  \label{schr9q}
\end{equation}%
\begin{equation}
\partial _{t}\mathbf{P+\partial }_{t}\mathbf{\mathbf{P}}^{\prime }\mathbf{+}%
\partial _{i}T^{ij}\mathbf{+}\partial _{i}T^{ij\prime }=\mathbf{f}+\mathbf{f}%
^{\prime }  \label{momNLSq}
\end{equation}%
where charge density $\rho $, current density $\mathbf{J}$ and tensor
elements $T^{ij}$ are defined by (\ref{schr8}),(\ref{schr8a}) and (\ref{Tij}%
) in terms of given functions $\psi ,\varphi ,\mathbf{A}$ \ and $\rho
^{\prime },\mathbf{J}^{\prime }$ $\mathbf{\mathbf{P}}^{\prime }$ $%
T^{ij\prime }$, and quantities $\mathbf{f}^{\prime }$ are perturbation terms
which vanish as $n\rightarrow \infty $. \ We show now how to modify the
proofs of statements in Section \ref{SNLSpoint} on concentrating solutions
to obtain similar statements on concentrating asymptotic solutions.

In the proof of Lemma \ref{Lmomf} \ equation (\ref{momNLSf}) is replaced by (%
\ref{momNLSq}). This leads to replacement of the equation (\ref{Peq}) by the
equation%
\begin{gather}
\int_{\Omega \left( \mathbf{\hat{r}}(t),R_{n}\right) }\mathbf{P}\left(
t\right) \,\mathrm{d}^{3}x-\int_{\Omega \left( \mathbf{\hat{r}}%
(t_{0}),R_{n}\right) }\mathbf{P}\left( t_{0}\right) \,\mathrm{d}%
^{3}x-Q_{01}\,  \label{Peqq} \\
-\int_{t_{0}}^{t}\int_{\Omega _{n}}\mathbf{f}\mathrm{d}\mathbf{x}\mathrm{d}%
t^{\prime }\,+Q_{0}+Q_{0}^{\prime }=0,  \notag
\end{gather}%
with 
\begin{gather}
Q_{0}^{\prime }=\int_{\Omega \left( \mathbf{\hat{r}}(t),R_{n}\right) }%
\mathbf{P}^{\prime }\left( t\right) \,\mathrm{d}^{3}x-\int_{\Omega \left( 
\mathbf{\hat{r}}(t_{0}),R_{n}\right) }\mathbf{P}^{\prime }\left(
t_{0}\right) \,\mathrm{d}^{3}x  \label{Q0p} \\
-\int_{t_{0}}^{t}\int_{\partial \Omega _{n}}\mathbf{P}^{\prime }\mathbf{\hat{%
v}}\cdot \mathbf{\bar{n}}\mathrm{d}\sigma \mathrm{d}t^{\prime
}+\int_{t_{0}}^{t}\int_{\Omega _{n}}\partial _{i}T^{ij\prime }\,\mathrm{d}%
\mathbf{x}\mathrm{d}t^{\prime }-\int_{t_{0}}^{t}\int_{\Omega _{n}}\mathbf{f}%
^{\prime }\,\mathrm{d}\mathbf{x}\mathrm{d}t^{\prime }.  \notag
\end{gather}

If%
\begin{equation}
Q_{0}^{\prime }\rightarrow 0  \label{Q0pp0}
\end{equation}%
the proof of Lemma \ref{Lmomf} remains valid with $Q_{00}=Q_{0}+Q_{0}^{%
\prime }-Q_{01}$, and we obtain the following lemma: \ 

\begin{lemma}
\label{Lmomfq}Let the conservation law (\ref{momNLS}) be replaced by (\ref%
{momNLSq})\ with condition (\ref{Q0pp0}) satisfied. Then the statement of
Lemma \ref{Lmomf} \ remains true: 
\begin{equation}
\int_{\Omega _{n}}\mathbf{P}_{n}\left( t\right) \,\mathrm{d}\mathbf{x}%
=\int_{t_{0}}^{t}\int_{\Omega _{n}}\mathbf{f}\,\mathrm{d}\mathbf{x}\mathrm{d}%
t^{\prime }+\mathbf{p}_{\infty }+Q_{00}  \label{intptq}
\end{equation}%
where 
\begin{equation}
Q_{00}\rightarrow 0\text{ \ \ as \ }n\rightarrow \infty
\end{equation}%
uniformly on $\left[ T_{-},T_{+}\right] $.
\end{lemma}

Sufficient conditions for fulfillment of (\ref{Q0pp0}) are the following
limit relations 
\begin{equation}
\int_{\Omega \left( \mathbf{\hat{r}}(t),R_{n}\right) }\mathbf{P}^{\prime
}\left( t\right) \,\mathrm{d}\mathbf{x}\rightarrow 0,  \label{Ppto0}
\end{equation}%
\begin{equation}
\int_{t_{0}}^{t}\int_{\partial \Omega _{n}}\mathbf{P}^{\prime }\mathbf{\hat{v%
}}\cdot \mathbf{\bar{n}}\mathrm{d}\sigma \mathrm{d}t^{\prime }\rightarrow 0,
\label{Ppdto0}
\end{equation}%
\begin{equation}
\int_{t_{0}}^{t}\int_{\partial \Omega _{n}}\mathbf{\bar{n}}_{i}T^{ij\prime
}\,\mathrm{d}\sigma \mathrm{d}t^{\prime }\rightarrow 0,  \label{Tijpto0}
\end{equation}%
\begin{equation}
\int_{t_{0}}^{t}\int_{\Omega _{n}}\,\mathbf{f}^{\prime }\,\mathrm{d}\mathbf{x%
}\mathrm{d}t^{\prime }\rightarrow 0  \label{florto0}
\end{equation}%
uniformly on the time interval $\left[ T_{-},T_{+}\right] .$

Let us take a look at the proof of Lemma \ref{LJlimv}. Since we replace the
continuity equation \ (\ref{schr9}) by (\ref{schr9q}), the equation (\ref%
{Pintr0}) involves now additional terms: 
\begin{equation}
\int_{\Omega _{n}}\partial _{t}\left( \left( \mathbf{x-r}\right) \rho
\right) \mathrm{d}\mathbf{x}+\partial _{t}\mathbf{r}\int_{\Omega _{n}}\rho 
\mathrm{d}\mathbf{x}+\int_{\partial \Omega _{n}}\left( \mathbf{x-r}\right) 
\mathbf{n}\cdot \mathbf{J}\mathrm{d}\mathbf{x}+Q_{2}^{\prime }=\int_{\Omega
_{n}}\mathbf{J}\mathrm{d}\mathbf{x},
\end{equation}%
with 
\begin{equation}
Q_{2}^{\prime }=\int_{\Omega _{n}}\partial _{t}\left( \left( \mathbf{x-r}%
\right) \rho ^{\prime }\right) \mathrm{d}\mathbf{x}+\partial _{t}\mathbf{r}%
\int_{\Omega _{n}}\rho ^{\prime }\mathrm{d}\mathbf{x}+\int_{\partial \Omega
_{n}}\left( \mathbf{x-r}\right) \mathbf{n}\cdot \mathbf{J}^{\prime }\mathrm{d%
}\mathbf{x}-\int_{\Omega _{n}}\mathbf{J}^{\prime }\mathrm{d}\mathbf{x.}
\label{Q2p}
\end{equation}

We assume that 
\begin{equation}
Q_{2}^{\prime }\rightarrow 0  \label{Q2pto0}
\end{equation}%
and conclude that the statement of of Lemma \ref{LJlimv} holds:

\begin{lemma}
\label{LJlimvq}Let the continuity equation \ (\ref{schr9}) be replaced by (%
\ref{schr9q}) with condition (\ref{Q2pto0}) fulfilled. Then the statement of
Lemma \ref{LJlimv} is true.
\end{lemma}

Note that according to (\ref{dtxrh}) sufficient conditions for fulfillment
of (\ref{Q2pto0}) \ are as follows:%
\begin{equation}
\int_{\Omega _{n}}\rho ^{\prime }\mathrm{d}\mathbf{x}\rightarrow 0,
\label{Q2p1}
\end{equation}%
\begin{equation}
\int_{\partial \Omega _{n}}\left( \mathbf{x-r}\right) \mathbf{n}\cdot 
\mathbf{J}^{\prime }\mathrm{d}\mathbf{x}\rightarrow 0,  \label{Q2p2}
\end{equation}%
\begin{equation}
\int_{\partial \Omega _{n}}\left( \mathbf{x}-\mathbf{r}\right) \mathbf{\hat{v%
}}\cdot \mathbf{\bar{n}}\rho ^{\prime }\mathrm{d}\sigma \rightarrow 0.
\label{Q2p3}
\end{equation}

Now we consider the proof of Lemma \ref{Ladjch}.\ Equation (\ref{dtrb}) is
replaced by 
\begin{equation}
\bar{\rho}\left( t\right) -\bar{\rho}\left( t_{0}\right)
-\int_{t_{0}}^{t}\int_{\partial \Omega _{n}}\mathbf{\hat{v}\cdot n}\rho 
\mathrm{d}\mathbf{x}\mathrm{d}t^{\prime }+\int_{t_{0}}^{t}\int_{\partial
\Omega _{n}}\mathbf{n}\cdot \mathbf{J}\mathrm{d}\mathbf{x}\mathrm{d}%
t^{\prime }+Q_{3}^{\prime }=0,  \label{dtrbq}
\end{equation}%
where%
\begin{equation}
Q_{3}^{\prime }=\bar{\rho}^{\prime }\left( t\right) -\bar{\rho}^{\prime
}\left( t_{0}\right) -\int_{\partial \Omega _{n}}\mathbf{\hat{v}\cdot n}\rho
^{\prime }\mathrm{d}\mathbf{x}+\int_{\partial \Omega _{n}}\mathbf{n}\cdot 
\mathbf{J}^{\prime }\mathrm{d}\mathbf{x.}  \label{Q3p}
\end{equation}%
The proof of Lemma \ \ref{Ladjch} is preserved if we assume that 
\begin{equation}
Q_{3}^{\prime }\rightarrow 0.  \label{Q3pto0}
\end{equation}

\begin{lemma}
Let the continuity equation \ (\ref{schr9}) be replaced by (\ref{schr9q})
with condition (\ref{Q3pto0}) fulfilled. Then the statement of Lemma \ref%
{Ladjch} is true.
\end{lemma}

Now we collect all the assumptions required for exact solutions in Section %
\ref{SNLSpoint} that remain valid for asymptotic solutions in the following
definition:

\begin{definition}[Concentrating asymptotic solutions]
\label{Dconcas}We assume all conditions of Definition \ref{Dlocconverges}
except condition (i), which is replaced by the following weaker condition:
conservation laws (\ref{momNLSq}) and (\ref{schr9q}) are fulfilled and
conditions (\ref{Q2pto0}), (\ref{Q0pp0}), (\ref{Q3pto0}) hold. Then we say
that asymptotic solutions of\ the NLS\ equation (\ref{NLS}) concentrate at $%
\mathbf{\hat{r}}(t)$. We call $\mathbf{\hat{r}}(t)$ a concentration
trajectory \ of the NLS equation in asymptotic sense if there exists a
sequence of asymptotic solutions which concentrates at $\mathbf{\hat{r}}(t)$.
\end{definition}

Using Lemmas \ref{Lmomfq} and \ref{LJlimvq} instead of Lemmas \ref{Lmomf}
and \ref{LJlimv} we obtain statements of Theorems \ref{TconcNLS} and \ref%
{Cq0s} \ under the assumption that asymptotic solutions $\psi $ of\ the NLS\
equation (\ref{NLS}) concentrate at $\mathbf{\hat{r}}(t)$. In particular, we
obtain

\begin{theorem}
\label{Cq0sa}Assume that potentials $\varphi \left( t,\mathbf{x}\right) $, $%
\mathbf{A}\left( t,\mathbf{x}\right) $ are defined and twice continuously
differentiable in a domain $D\subset \mathbb{R}\times \mathbb{R}^{3}$, \ the
trajectory $\left( t,\mathbf{\hat{r}}\left( t\right) \right) $ lies in this
domain and the limit potentials $\varphi _{\infty }$ , $\mathbf{A}_{\infty }$
\ are the restriction of fixed potentials $\varphi $,$\mathbf{A}$ as in (\ref%
{fiarestr}). Let EM fields \ $\mathbf{E}\left( t,\mathbf{x}\right) ,\mathbf{B%
}\left( t,\mathbf{x}\right) $ \ be determined in terms of the potentials by
formula (\ref{maxw3a}).\ Let asymptotic solutions $\psi $ of\ the NLS\
equation (\ref{NLS}) asymptotically concentrate at $\mathbf{\hat{r}}(t)$.
Then the trajectory $\mathbf{\hat{r}}$ satisfies Newton's law of motion (\ref%
{New2sa}).
\end{theorem}

\subsection{Point trajectories as trajectories of asymptotic concentration 
\label{Spoint}}

Let us consider NLS quation (\ref{NLS}) in a domain $D\subset \mathbb{R}%
\times \mathbb{R}^{3}$ and assume that potentials $\varphi \left( t,\mathbf{x%
}\right) $ and $\mathbf{A}\left( t,\mathbf{x}\right) $ are defined and twice
continuously differentiable in domain $D$. Let us consider equations (\ref%
{flor})-(\ref{flor1}) describing Newtonian dynamics of a point charge in EM\
field. Let us consider a solution $\mathbf{r}(t)$ of equations (\ref{flor})-(%
\ref{flor1}), and assume that the trajectory \ $\left( t,\mathbf{r}%
(t)\right) $ lies in $D$ \ on the time interval $T_{-}\leq t\leq T_{+}$. \
Now we construct asymptotic solutions of the NLS which concentrate at $%
\mathbf{r}(t)$.

As a first step we find the linear part of $\varphi \left( t,\mathbf{x}%
\right) $, $\mathbf{A}\left( t,\mathbf{x}\right) $ at $\mathbf{r}(t)$ as in (%
\ref{fiarestr})%
\begin{eqnarray}
\varphi _{\infty }\left( t,\mathbf{x}\right) &=&\varphi \left( t,\mathbf{r}%
\right) +\left( \mathbf{x}-\mathbf{r}\right) \nabla \varphi \left( t,\mathbf{%
r}\right) ,  \label{fiarestra} \\
\mathbf{A}_{\infty }\left( t,\mathbf{x}\right) &=&\mathbf{A}\left( t,\mathbf{%
r}\right) +\left( \mathbf{x}-\mathbf{r}\right) \nabla \mathbf{A}\left( t,%
\mathbf{r}\right) ,  \notag
\end{eqnarray}%
As a second step we construct the auxiliary potentials 
\begin{gather}
\varphi _{\mathrm{aux}}\left( t,\mathbf{x}\right) =\varphi \left( t,\mathbf{r%
}\right) +\left( \mathbf{x}-\mathbf{r}\right) \nabla \varphi \left( t,%
\mathbf{r}\right) +\varphi _{2}\left( t,\mathbf{x}-\mathbf{r}\right) ,
\label{fiaux} \\
\mathbf{A}_{\mathrm{aux}}\left( t,\mathbf{x}\right) =\mathbf{A}\left( t,%
\mathbf{r}\right) +\left( \mathbf{x}-\mathbf{r}\right) \nabla \mathbf{A}%
\left( t,\mathbf{r}\right) ,  \notag
\end{gather}%
where $\varphi _{2}$ is determined by (\ref{fiex2a}) and $\varphi _{2}\left(
t,\mathbf{x}-\mathbf{r}\right) $ is quadratic with respect to $\left( 
\mathbf{x}-\mathbf{r}\right) $.

The wave-corpuscle $\psi $ described in Theorem \ref{Twcem} is a solution of
the NLS equation with the potentials $\left( \varphi _{\mathrm{aux}},\mathbf{%
A}_{\mathrm{aux}}\right) $, and hence it exactly satisfies the corresponding
conservation laws. From fulfillment of (\ref{schr9}) and (\ref{momNLS}) for $%
\mathbf{P}\left( \varphi _{\mathrm{aux}},\mathbf{A}_{\mathrm{aux}}\right) $%
\textbf{, }$\mathbf{J}\left( \varphi _{\mathrm{aux}},\mathbf{A}_{\mathrm{aux}%
}\right) $\textbf{, }$T^{ij}\left( \varphi _{\mathrm{aux}},\mathbf{A}_{%
\mathrm{aux}}\right) $, $\mathbf{f}\left( \varphi _{\mathrm{aux}},\mathbf{A}%
_{\mathrm{aux}}\right) $ we obtain fulfillment of (\ref{schr9q}), (\ref%
{momNLSq}) with%
\begin{equation}
\mathbf{P}^{\prime }=\mathbf{P}\left( \varphi ,\mathbf{A}\right) -\mathbf{P}%
\left( \varphi _{\mathrm{aux}},\mathbf{A}_{\mathrm{aux}}\right) ,  \label{pp}
\end{equation}%
\begin{equation}
\mathbf{J}^{\prime }=\mathbf{J}\left( \varphi ,\mathbf{A}\right) -\mathbf{J}%
\left( \varphi _{\mathrm{aux}},\mathbf{A}_{\mathrm{aux}}\right) ,  \label{jp}
\end{equation}%
\begin{equation}
\mathbf{f}^{\prime }=\mathbf{f}\left( \varphi ,\mathbf{A}\right) -\mathbf{f}%
\left( \varphi _{\mathrm{aux}},\mathbf{A}_{\mathrm{aux}}\right) ,  \label{fp}
\end{equation}%
\begin{equation}
T^{ij\prime }=T^{ij}\left( \varphi ,\mathbf{A}\right) -T^{ij}\left( \varphi
_{\mathrm{aux}},\mathbf{A}_{\mathrm{aux}}\right) .  \label{tjp}
\end{equation}%
The expression for $\rho $ does not depend on the potentials, therefore 
\begin{equation}
\rho ^{\prime }=0.  \label{rp0}
\end{equation}%
Now we need to verify that conditions (\ref{Q2pto0}), (\ref{Q0pp0}), (\ref%
{Q3pto0}) hold for the wave-corpuscle $\psi $.

From (\ref{schr5}), (\ref{Pim}) we see that 
\begin{equation}
\mathbf{P}\left( \varphi ,\mathbf{A}\right) -\mathbf{P}\left( \varphi _{%
\mathrm{aux}},\mathbf{A}_{\mathrm{aux}}\right) =-\frac{q}{\mathrm{c}}\left( 
\mathbf{A-A}_{\mathrm{aux}}\right) \psi ^{\ast }\psi ,  \label{amaaux}
\end{equation}%
and according to (\ref{PmqJ}) 
\begin{equation}
\mathbf{J}\left( \varphi ,\mathbf{A}\right) -\mathbf{J}\left( \varphi _{%
\mathrm{aux}},\mathbf{A}_{\mathrm{aux}}\right) =-\frac{q^{2}}{m\mathrm{c}}%
\left( \mathbf{A-A}_{\mathrm{aux}}\right) \psi ^{\ast }\psi .  \label{Jp}
\end{equation}%
\ From the expression for the Lorentz density (\ref{flors}) we obtain 
\begin{gather*}
\mathbf{f}\left( \varphi ,\mathbf{A}\right) -\mathbf{f}\left( \varphi _{%
\mathrm{aux}},\mathbf{A}_{\mathrm{aux}}\right) =\rho \left( \mathbf{E-E}_{%
\mathrm{aux}}\right) +\frac{1}{\mathrm{c}}\left( \mathbf{J}\times \mathbf{B-J%
}_{\mathrm{aux}}\times \mathbf{B}_{\mathrm{aux}}\right) = \\
=\rho \left( \mathbf{E-E}_{\mathrm{aux}}\right) +\frac{1}{\mathrm{c}}\left(
\left( \mathbf{J-J}_{\mathrm{aux}}\right) \times \mathbf{B+J}_{\mathrm{aux}%
}\times \left( \mathbf{B}-\mathbf{B}_{\mathrm{aux}}\right) \right) .
\end{gather*}%
Using (\ref{Pim}) and (\ref{PmqJ}) we rewrite expression for $\mathbf{f}%
^{\prime }$ in the form 
\begin{gather}
\mathbf{f}^{\prime }=\rho \left( \mathbf{E-E}_{\mathrm{aux}}\right)
\label{florpa} \\
+\frac{1}{\mathrm{c}}\left( -\frac{q^{2}}{m\mathrm{c}}\mathring{\psi}%
^{2}\left( \mathbf{A-A}_{\mathrm{aux}}\right) \times \mathbf{B+}\chi \frac{q%
}{m}\mathring{\psi}^{2}\left[ \nabla S-\frac{q}{\chi \mathrm{c}}\mathbf{A}_{%
\mathrm{aux}}\right] \times \left( \mathbf{B}-\mathbf{B}_{\mathrm{aux}%
}\right) \right) .  \notag
\end{gather}%
To estimate the difference of tensor elements (\ref{tjp}) we use (\ref{Tij})
and (\ref{emfr10}). Note that according to the construction of $\varphi _{%
\mathrm{aux}}$ and $\mathbf{A}_{\mathrm{aux}}$ the differences $\varphi
-\varphi _{\mathrm{aux}}$ and $\ \mathbf{A-A}_{\mathrm{aux}}$ have the
second order zero at $\mathbf{r}(t)$. Hence the diffferences $\mathbf{E-E}_{%
\mathrm{aux}}$ and $\mathbf{B-B}_{\mathrm{aux}}$ have a zero of the first
order at $\mathbf{x}=\mathbf{r}$:%
\begin{equation}
\left\vert \mathbf{A-A}_{\mathrm{aux}}\right\vert \leq C\left\vert \mathbf{x}%
-\mathbf{r}\right\vert ^{2},\qquad \left\vert \varphi -\varphi _{\mathrm{aux}%
}\right\vert \leq C\left\vert \mathbf{x}-\mathbf{r}\right\vert ^{2},
\label{ama2}
\end{equation}%
\begin{equation}
\left\vert \mathbf{E-E}_{\mathrm{aux}}\right\vert \leq C\left\vert \mathbf{x}%
-\mathbf{r}\right\vert ,\qquad \left\vert \mathbf{B-B}_{\mathrm{aux}%
}\right\vert \leq C\left\vert \mathbf{x}-\mathbf{r}\right\vert ,  \label{eme}
\end{equation}%
and they are vanishingly small in $\Omega \left( \mathbf{r}(t),R_{n}\right) $%
. Now we estimate terms which enter (\ref{Q0p}), (\ref{Q2p}), (\ref{Q3p}).

\begin{lemma}
Let $\mathbf{P}^{\prime },\mathbf{J}^{\prime },\mathbf{f}^{\prime
},T^{ij\prime }$ be defined by (\ref{pp})-(\ref{tjp}), conditions (\ref{psi1}%
), (\ref{rpsi1}) and (\ref{ar4}) fulfilled. Then (\ref{Ppto0})-(\ref{florto0}%
) and (\ref{Q2p1})-(\ref{Q2p3}) and (\ref{Q3pto0}) hold.
\end{lemma}

\begin{proof}
The proof is based on inequalities (\ref{eme}) and (\ref{ama2}). To obtain (%
\ref{Ppto0}) we use (\ref{amaaux}) and (\ref{limrho}):%
\begin{gather*}
\left\vert \int_{\Omega \left( \mathbf{r}(t),R_{n}\right) }\mathbf{P}%
^{\prime }\left( t\right) \,\mathrm{d}\mathbf{x}\right\vert \leq
C\int_{\Omega \left( \mathbf{r}(t),R_{n}\right) }\left\vert \mathbf{x}-%
\mathbf{r}\right\vert ^{2}\mathring{\psi}_{a}^{2}\left( \left\vert \mathbf{x}%
-\mathbf{r}\right\vert \right) \,\mathrm{d}\mathbf{x} \\
\leq CR_{n}^{2}\int_{\Omega \left( \mathbf{0},R_{n}/a\right) }\mathring{\psi}%
_{a}^{2}\left( \left\vert \mathbf{y}\right\vert \right) \,\mathrm{d}\mathbf{y%
}\leq C_{1}R_{n}^{2},
\end{gather*}%
and we obtain (\ref{Ppto0}) for $R_{n}\rightarrow 0$. To \ obtain (\ref%
{Ppdto0}) we observe that according to \ (\ref{amaaux}) 
\begin{gather}
\left\vert \int_{\partial \Omega \left( \mathbf{r}(t),R_{n}\right) }\mathbf{P%
}^{\prime }\mathbf{\hat{v}}\cdot \mathbf{\bar{n}}\mathrm{d}\sigma \right\vert
\label{ipdom} \\
\leq C\int_{\partial \Omega \left( \mathbf{r}(t),R_{n}\right) }\left\vert 
\mathbf{x}-\mathbf{r}\right\vert ^{2}\mathring{\psi}_{a}^{2}\left(
\left\vert \mathbf{x}-\mathbf{r}\right\vert \right) \mathrm{d}\sigma =4\pi
CR_{n}^{4}a^{-3}\mathring{\psi}_{a}^{2}\left( R_{n}/a\right)  \notag \\
=4\pi CR_{n}\theta ^{3}\mathring{\psi}_{a}^{2}\left( \theta \right)
\rightarrow 0  \notag
\end{gather}%
\ Similarly we obtain \ and (\ref{Jp}). A straightforward estimate of (\ref%
{florpa}) yields (\ref{florto0}). Note that according to (\ref{emfr10}) and (%
\ref{emfr8}) the terms in $T^{ii}$ involving $G$ do not depend on $\left(
\varphi ,\mathbf{A}\right) $, and hence $G$ does not enter $T^{ij\prime }$.\
According to (\ref{Tij}), (\ref{emfr10}), (\ref{psdtps}) and (\ref{covNLS}) $%
T^{ij}\left( \varphi ,\mathbf{A}\right) $ is a quadratic function of
potentials and $T^{ij\prime }$ equals sum of terms, every of which involves
factors $\mathbf{A-A}_{\mathrm{aux}}\ $or \ $\varphi -\varphi _{\mathrm{aux}%
} $ which satisfy (\ref{ama2}) and also factors $\psi \nabla \psi ^{\ast }$
or $\nabla \psi \psi ^{\ast }$ \ or $\psi \psi ^{\ast }$. The factors $%
\mathbf{A-A}_{\mathrm{aux}}$ or \ $\varphi -\varphi _{\mathrm{aux}}$\ in $%
\Omega \left( \mathbf{r}(t),R_{n}\right) $ produce coefficient $CR_{n}^{2}$.
All the terms are easily estimated, for example an elementary inequality 
\begin{equation*}
\int_{\partial \Omega \left( \mathbf{r}(t),R_{n}\right) }\left\vert 
\mathring{\psi}_{a}\right\vert \left\vert \nabla \mathring{\psi}%
_{a}\right\vert \mathrm{d}\sigma \leq \frac{1}{2}\int_{\partial \Omega
\left( \mathbf{r}(t),R_{n}\right) }\left\vert \mathring{\psi}_{a}\right\vert
^{2}+\left\vert \nabla \mathring{\psi}_{a}\right\vert ^{2}\mathrm{d}\sigma
\end{equation*}%
allows to use (\ref{intdom})-(\ref{intdom0}), and we obtain that 
\begin{equation*}
\int_{\partial \Omega \left( \mathbf{r}(t),R_{n}\right) }\left\vert 
\mathring{\psi}_{a}\right\vert \left\vert \nabla \mathring{\psi}%
_{a}\right\vert \mathrm{d}\sigma \rightarrow 0.
\end{equation*}%
Therefore (\ref{Tijpto0}) is fulfilled. \ Since $\rho ^{\prime }=0$ \ (\ref%
{Q2p1}) and (\ref{Q2p3}) are fulfilled, (\ref{Q2p2}) follows from \ref{Ppto0}%
) and (\ref{PmqJ}). To check condition (\ref{Q3pto0}) we note that $\rho
^{\prime }=0$ and after using (\ref{PmqJ}) we estimate the surface integral
in (\ref{Q3p}) involving $\mathbf{J}^{\prime }$\ similarly to (\ref{ipdom}).
\end{proof}

\begin{theorem}
\label{TCq0sq}Let potentials $\varphi \left( t,\mathbf{x}\right) $, $\mathbf{%
A}\left( t,\mathbf{x}\right) $ be twice continuously differentiable in a
domain $D$. Let form factor $\mathring{\psi}_{1}$ satisfy conditions (\ref%
{psi1}), (\ref{rpsi1}). Let sequences $a_{n},R_{n}$ satisfy (\ref{ar4}). Let 
$\mathbf{r}(t)$ be a solution of \ equations (\ref{flor})-(\ref{flor1}) with
trajectory in $D$. \ Then wave-corpuscles constructed \ in Theorem \ref%
{Twcconc} based on $\mathbf{r}(t)$ and potentials $\varphi _{\mathrm{aux}},%
\mathbf{A}_{\mathrm{aux}}$ given by (\ref{fiaux}) provide asymptotic
solutions in the sense of Definition \ref{Dconcas} that concentrate at $%
\mathbf{r}(t)$.
\end{theorem}

As a corollary we obtain the following theorem.

\begin{theorem}
\label{Ttrconv}Let $\varphi \left( t,\mathbf{x}\right) $, $\mathbf{A}\left(
t,\mathbf{x}\right) $ are defined and twice continuously differentiable in a
domain $D$. Then a trajectory $\mathbf{r}(t)$ which lies in $D$ is a
concentration trajectory of asymptotic solutions of NLS equation (\ref{NLS})
if and only if \ it satisfies Newton's law (\ref{flor}) with the Lorentz
force defined by (\ref{flor1}).
\end{theorem}

Since asymptotic solutions in the sense of Definition \ref{Dconcas} \ in
Theorem \ref{TCq0sq} are constructed as wave-corpuscles, we obtain the
following corollary.

\begin{corollary}
\label{Cwccon}If a sequence of concentrating asymptotic solutions
concentrates at a trajectory $\mathbf{r}(t)$, then there exists a sequence
of concentrating wave-corpuscle asypmptotic solutions which concentrate at
the same trajectory.
\end{corollary}

\section{Appendix 1: Lagrangian field formalism for nonlinear Schr\"{o}%
dinger Equation \ \label{sschrodinger}}

To derive the conservation laws for the NLS equation it is convenient to use
relativistic notation. We introduce the following 4-differential operators \ 
\begin{equation}
\partial _{\mu }=\left( \frac{1}{\mathrm{c}}\partial _{t},\nabla \right) ,\
\partial ^{\mu }=\left( \frac{1}{\mathrm{c}}\partial _{t},-\nabla \right) ,
\end{equation}%
where the indices $\mu $ \ take four values \ $\mu =0,1,2,3$ \ and count
components of the right-hand sides; in particular $\partial _{0}=\frac{1}{%
\mathrm{c}}\partial _{t}$. We also denote 
\begin{equation}
\psi ^{,\mu }=\partial ^{\mu }\psi ,\qquad \psi _{,\mu }=\partial _{\mu
}\psi .
\end{equation}%
The EM 4-potential is given by the formula 
\begin{equation}
A=A^{\mu }=\left( \varphi ,\mathbf{A}\right) ,  \label{maxw2a}
\end{equation}%
where $\varphi \left( t,\mathbf{x}\right) ,\mathbf{A}\left( t,\mathbf{x}%
\right) $ are given potentials, and the EM power tensor \ is defined as
follows: 
\begin{equation}
F^{\mu \nu }=\partial ^{\mu }A^{\nu }-\partial ^{\nu }A^{\mu }.
\label{maxw2b}
\end{equation}%
The covariant derivatives which involve the EM fields are defined by (\ref%
{covNLS}), and corresponding 4-differential operators have the form 
\begin{equation}
\tilde{\partial}_{\mu }=\left( \frac{1}{\mathrm{c}}\tilde{\partial}_{t},%
\tilde{\nabla}\right) ,\ \tilde{\partial}^{\mu }=\left( \frac{1}{\mathrm{c}}%
\tilde{\partial}_{t},-\tilde{\nabla}\right) .  \label{maxw2c}
\end{equation}%
We also use the following notation for covariant derivatives%
\begin{equation}
\psi ^{;\mu }=\tilde{\partial}^{\mu }\psi ,\qquad \psi ^{;\mu \ast }=\tilde{%
\partial}^{\mu \ast }\psi ^{\ast },\qquad \psi _{;\mu }=\tilde{\partial}%
_{\mu }\psi ,\qquad \psi _{\ell ;\mu }^{\ast }=\tilde{\partial}_{\mu }^{\ast
}\psi ^{\ast }.  \label{flagr6b}
\end{equation}%
Obviously,%
\begin{equation}
\psi ^{;\mu }=\psi ^{,\mu }+\frac{\mathrm{i}q}{\chi \mathrm{c}}A^{\mu }\psi .
\label{comsem}
\end{equation}%
The NLS equation (\ref{NLS}) is the Euler-Lagrange field equation for the
following Lagrangian density: 
\begin{equation}
L=\mathrm{i}\frac{\chi }{2}\left[ \psi ^{\ast }\tilde{\partial}_{t}\psi
-\psi \tilde{\partial}_{t}^{\ast }\psi ^{\ast }\right] -\frac{\chi ^{2}}{2m}%
\left[ \tilde{\nabla}\psi \tilde{\nabla}^{\ast }\psi ^{\ast }+G\left( \psi
^{\ast }\psi \right) \right] \   \label{LNLS}
\end{equation}%
where $G\left( s\right) $ is given by (\ref{Gan}). The 4-current $J^{\nu }$
for the Lagrangian is defined by the formula 
\begin{equation}
J^{\nu }=-\mathrm{i}\frac{q}{\chi }\left( \frac{\partial L}{\partial \psi
_{;\nu }}\psi -\frac{\partial L}{\partial \psi _{;\nu }^{\ast }}\psi ^{\ast
}\right) ,  \label{flagr9}
\end{equation}%
it can be written in the form 
\begin{equation}
J=J^{\nu }=\left( \mathrm{c}\rho ,\mathbf{J}\right)  \label{maxw2a1}
\end{equation}%
where $\rho ,\mathbf{J}$ \ are given by (\ref{schr8}),(\ref{schr8a}). \ The
Lagrangian (\ref{LNLS}) and the NLS equation (\ref{NLS}) are gauge
invariant, that is invariant with respect to the multiplication of $\psi $
by $\mathrm{e}^{\mathrm{i}\gamma }$ with real $\gamma $: 
\begin{equation}
L\left( \mathrm{e}^{\mathrm{i}\gamma }\psi ,\mathrm{e}^{\mathrm{i}\gamma
}\psi _{;\mu },\mathrm{e}^{-\mathrm{i}\gamma }\psi ^{\ast },\mathrm{e}^{-%
\mathrm{i}\gamma }\psi _{\ell ;\mu }^{\ast }\right) =L\left( \psi ,\psi
_{;\mu },\psi ^{\ast },\psi _{;\mu }^{\ast }\right) .  \label{eisinv}
\end{equation}%
If we take derivative of the above conditon (\ref{eisinv}) with respect to $%
\ \gamma $ at $\gamma =0,$ we obtain the \ following structural restriction
on the Lagrangian $L$: 
\begin{equation}
\frac{\partial L}{\partial \psi _{;\mu }}\psi _{;\mu }-\frac{\partial L}{%
\partial \psi _{;\mu }^{\ast }}\psi _{;\mu }^{\ast }+\frac{\partial L}{%
\partial \psi }\psi _{\ell }-\frac{\partial L}{\partial \psi ^{\ast }}\psi
_{\ell }^{\ast }=0.  \label{eisinv0}
\end{equation}%
Direct verification shows that if a Lagrangian $L$ satisfies the above
structural condition, then the current defined by (\ref{flagr9}) for a
solution of the Euler-Lagrange field equation (NLS in our case) satisfies
the continuity equation%
\begin{equation}
\partial _{\nu }J^{\nu }=0\text{ }  \label{flagr11}
\end{equation}%
which can be written in the form (\ref{schr9}). \ 

We introduce the \emph{\ energy-momentum tensor} (EnMT) (see \cite{Barut}
for the general theory) for the NLS by the following formula: 
\begin{equation}
T^{\mu \nu }=\frac{\partial L}{\partial \psi _{;\mu }}\psi ^{;\nu }+\frac{%
\partial L}{\partial \psi _{;\mu }^{\ast }}\psi ^{;\nu \ast }-g^{\mu \nu }L,
\label{enmom5}
\end{equation}%
where $g^{\mu \nu }$ is the Minkowski metric tensor, that is%
\begin{equation}
g^{\mu \nu }=0\text{ for }\mu \neq \nu ,\quad g^{00}=1,\quad \text{and }%
g^{jj}=-1\text{ for }j=1,2,3.
\end{equation}

\begin{proposition}
Let EnMT $T^{\mu \nu }$\ be defined by formula (\ref{enmom5}) $\ $for a
solution of the NLS equation (\ref{NLS}). Then $T^{\mu \nu }$ satisfies the
following EnMT conservation law 
\begin{equation}
\partial _{\mu }T^{\mu \nu }=f^{\nu }  \label{contmnlaw}
\end{equation}%
where $f^{\nu }$ is the \emph{Lorentz force density} defined by the formula 
\begin{equation}
f^{\nu }=\frac{1}{\mathrm{c}}J_{\mu }F^{\nu \mu }  \label{flord}
\end{equation}%
with $J_{\mu }$ and $\ F^{\nu \mu }$ defined by respectively by (\ref{flagr9}%
) and (\ref{maxw2b}).
\end{proposition}

\begin{proof}
According to (\ref{comsem}) $\ $and (\ref{flagr9} the definition (\ref%
{enmom5}) can be written in the form 
\begin{gather*}
T^{\mu \nu }=\frac{\partial L}{\partial \psi _{;\mu }}\psi ^{,\nu }+\frac{%
\partial L}{\partial \psi _{;\mu }}\frac{\mathrm{i}q}{\chi \mathrm{c}}A^{\nu
}\psi +\frac{\partial L}{\partial \psi _{;\mu }^{\ast }}\psi ^{,\nu \ast }-%
\frac{\partial L}{\partial \psi _{;\mu }^{\ast }}\frac{\mathrm{i}q}{\chi 
\mathrm{c}}A^{\nu }\psi ^{\ast }-g^{\mu \nu }L \\
=\frac{\partial L}{\partial \psi _{,\mu }}\psi ^{,\nu }+\frac{\partial L}{%
\partial \psi _{,\mu }^{\ast }}\psi ^{,\nu \ast }-\frac{1}{\mathrm{c}}J^{\mu
}A^{\nu }-g^{\mu \nu }L.
\end{gather*}%
where $J^{\nu }$ is defined by (\ref{flagr9}). We differentiate the above
expression \ and use the continuity equation (\ref{flagr11}):%
\begin{gather*}
\partial _{\mu }T^{\mu \nu }=\partial _{\mu }\left( \frac{\partial L}{%
\partial \psi _{;\mu }}\psi ^{,\nu }\right) +\partial _{\mu }\left( \frac{%
\partial L}{\partial \psi _{;\mu }^{\ast }}\psi ^{,\nu \ast }\right)
-\partial _{\mu }\left( J^{\mu }A^{\nu }\right) -\partial _{\mu }\left(
g^{\mu \nu }L\right) \\
=\partial _{\mu }\frac{\partial L}{\partial \psi _{;\mu }}\psi ^{,\nu }+%
\frac{\partial L}{\partial \psi _{;\mu }}\partial _{\mu }\psi ^{,\nu
}+\partial _{\mu }\frac{\partial L}{\partial \psi _{;\mu }^{\ast }}\psi
^{,\nu \ast }+\frac{\partial L}{\partial \psi _{;\mu }^{\ast }}\partial
_{\mu }\psi ^{,\nu \ast }-\frac{1}{\mathrm{c}}J^{\mu }\partial _{\mu }A^{\nu
}-g^{\mu \nu }\partial _{\mu }\left( L\right) .
\end{gather*}%
\ The NLS equation (\ref{NLS}) can be written in the form 
\begin{equation}
\ \frac{\partial L}{\partial \psi ^{\ast }}-\partial _{\mu }\frac{\partial L%
}{\partial \psi _{,\mu }^{\ast }}=0,  \label{flagr8c}
\end{equation}%
where the Lagrangian $L$ is considered as a function of $\psi ,\psi ^{\ast
},\psi _{,\mu }^{\ast }$ $\psi _{,\mu }^{\ast },$ and of the spatial and
time variables which enter through $A^{\nu }$. Using (\ref{flagr8c})
together with its conjugate \ we obtain that 
\begin{equation*}
\partial _{\mu }T^{\mu \nu }=-\frac{1}{\mathrm{c}}J^{\mu }\partial _{\mu
}A^{\nu }+\left[ \frac{\partial L}{\partial \psi }\partial ^{\nu }\psi +%
\frac{\partial L}{\partial \psi ^{\ast }}\partial ^{\nu }\psi ^{\ast }+\frac{%
\partial L}{\partial \psi _{,\mu }^{\ast }}\partial ^{\nu }\psi _{\mu
}^{\ast }+\frac{\partial L}{\partial \psi _{,\mu }}\partial ^{\nu }\psi
_{\mu }-\partial ^{\nu }\left( L\right) \right]
\end{equation*}%
Note that the expression in brackets equals the partial derivative $\partial
^{\nu }L$ (we denote $\partial ^{\nu }\left( \mathcal{L}\right) $ the
complete derivative of $\mathcal{L}$ and $\partial ^{\nu }\mathcal{L}$ the
partial one). Therefore 
\begin{equation*}
\partial _{\mu }T^{\mu \nu }=-\frac{1}{\mathrm{c}}J^{\mu }\partial _{\mu
}A^{\nu }+\partial ^{\nu }L,
\end{equation*}%
Note that the partial derivative $\partial ^{\nu }L$ can be evaluated as
follows: 
\begin{equation*}
\partial ^{\nu }L=\frac{\mathrm{i}q}{\chi \mathrm{c}}\left( \frac{\partial L%
}{\partial \psi _{;\mu }}\psi -\frac{\partial L}{\partial \psi _{;\mu
}^{\ast }}\psi ^{\ast }\right) \partial ^{\nu }A_{\mu }=-\frac{1}{\mathrm{c}}%
J^{\mu }\partial ^{\nu }A_{\mu }=-\frac{1}{\mathrm{c}}J_{\mu }\partial ^{\nu
}A^{\mu }.
\end{equation*}%
Hence, 
\begin{equation*}
\partial _{\mu }T^{\mu \nu }=-\frac{1}{\mathrm{c}}J_{\mu }\partial ^{\mu
}A^{\nu }+\frac{1}{\mathrm{c}}J_{\mu }\partial ^{\nu }A^{\mu },
\end{equation*}%
\ yielding the EnMT conservation law (\ref{contmnlaw}).\ 
\end{proof}

The entries of the EnMT can be interpreted as energy and momentum densities $%
u,p^{j}$ respectively, namely 
\begin{gather}
u=T^{00}=\frac{\chi ^{2}}{2m}\left[ \left\vert \tilde{\nabla}\psi
\right\vert ^{2}+G\left( \left\vert \psi \right\vert ^{2}\right) \right] ,
\label{emfr8} \\
p^{j}=T^{0j}=\mathrm{i}\frac{\chi }{2}\left( \psi \tilde{\partial}_{j}^{\ast
}\psi ^{\ast }-\psi ^{\ast }\tilde{\partial}_{j}\psi \right) .  \label{emfr9}
\end{gather}%
The formula for the momentum density can be written in the form (\ref{schr5}%
). \ The proportionality (\ref{PmqJ}) is a specific property of the NLS\
Lagrangian and does not hold in a general case. Remaining entries of the
EnMT \ take the form 
\begin{equation*}
T^{j0}=-\frac{\chi ^{2}}{2m}\left( \tilde{\partial}_{t}\psi \tilde{\partial}%
_{j}^{\ast }\psi ^{\ast }+\tilde{\partial}_{t}^{\ast }\psi ^{\ast }\tilde{%
\partial}_{j}\psi \right) ,\quad \ j=1,2,3,
\end{equation*}%
\begin{equation}
T^{ii}=u-\frac{\chi ^{2}}{m}\tilde{\partial}_{i}\psi \tilde{\partial}%
_{i}^{\ast }\psi ^{\ast }+\mathrm{i}\frac{\chi }{2}\left( \psi \tilde{%
\partial}_{t}^{\ast }\psi ^{\ast }-\psi ^{\ast }\tilde{\partial}_{t}\psi
\right) ,  \label{emfr10}
\end{equation}%
and for $i\neq j,\quad \ i,j=1,2,3.$%
\begin{equation}
T^{ij}=-\frac{\chi ^{2}}{2m}\left( \tilde{\partial}_{i}\psi \tilde{\partial}%
_{j}^{\ast }\psi ^{\ast }+\tilde{\partial}_{j}\psi \tilde{\partial}%
_{i}^{\ast }\psi ^{\ast }\right) .  \label{Tij}
\end{equation}%
Note that the Lorentz force density $f^{\nu }$ in (\ref{contmnlaw}) can be
written in the form 
\begin{equation}
f^{\nu }=\frac{1}{\mathrm{c}}J_{\mu }F^{\nu \mu }=\left( f^{0},\mathbf{f}%
\right) =\left( \frac{1}{\mathrm{c}}\mathbf{J}\cdot \mathbf{E},\rho \mathbf{E%
}+\frac{1}{\mathrm{c}}\mathbf{J}\times \mathbf{B}\right)  \label{emtn2b}
\end{equation}%
where, in particular, the momentum equation has the form (\ref{momNLS}).

\begin{remark}
Note that the derivation of the conservation law (\ref{contmnlaw}) is valid
for any local solution $\psi $ of the NLS equation which has continuous
second derivatives, $G$ is continuously differentiable on the range of $%
\left\vert \psi \right\vert ^{2}$ and the nonlinear term $G\left( \psi \psi
^{\ast }\right) $ which enters (\ref{LNLS}) \ has continuous first
derivatives. Note that the \ derivatives $G\left( \psi \psi ^{\ast }\right) $
\end{remark}

\section{Appendix 2: Splitting of a field into potential and sphere-tangent
parts \label{Svsplit}}

Here we describe splitting of a general vector field into sphere-tangent and
gradient components\ which was used in derivation of the properties of
wave-corpuscles.

\begin{lemma}
Any continuously differentiable vector field $\mathbf{V}\left( \mathbf{y}%
\right) $ can be uniquely splitted into a gradient field $\nabla P$\ and a
sphere-tangent field $\mathbf{\breve{V}}$ 
\begin{equation}
\mathbf{V}=\nabla P+\mathbf{\breve{V}.}  \label{vsplit}
\end{equation}%
where $P$ is continuously differentiable, $\mathbf{\breve{V}}$ is continuous
and satisfies the orthogonality condition (\ref{ahat0}), namely $\mathbf{%
\breve{V}}\cdot \mathbf{y=0}$.
\end{lemma}

\begin{proof}
To determine $P$ we multiply (\ref{vsplit}) by $\mathbf{y}$ and obtain \ 
\begin{equation*}
\mathbf{y}\cdot \mathbf{V}=\mathbf{y}\cdot \nabla P+\mathbf{y}\cdot \mathbf{%
\breve{V}=y}\cdot \nabla P
\end{equation*}%
The directional derivative $\mathbf{y}\cdot \nabla $ can be written as $r%
\mathbf{\partial }_{r}$ and \ using notation $\mathbf{y}=\mathbf{\Omega }r$
where\ $\left\vert \mathbf{\Omega }\right\vert =1$, $r=\left\vert \mathbf{y}%
\right\vert $\ we rewrite the above equation in the form 
\begin{equation*}
\mathbf{\Omega }\cdot \mathbf{V}=\mathbf{\partial }_{r}P\left( \mathbf{%
\Omega }r\right) .
\end{equation*}%
Note that $P$ \ is defined from the above equation up to a function $C\left( 
\mathbf{\Omega }\right) $; since the only function \ of this form which is
continuous at the origin must be a constant, $P$ \ is defined uniquely
modulo constants. We can find $P$ by integration: 
\begin{equation}
P\left( \mathbf{y}\right) =\int_{0}^{\left\vert \mathbf{y}\right\vert }%
\mathbf{\Omega }\cdot \mathbf{V}\left( \mathbf{\Omega }r\right) \,\mathrm{d}%
r=\int_{0}^{\left\vert \mathbf{y}\right\vert }\frac{1}{\left\vert \mathbf{y}%
\right\vert }\mathbf{y}\cdot \mathbf{V}\left( \frac{1}{\left\vert \mathbf{y}%
\right\vert }\mathbf{y}r\right) \,\mathrm{d}r.  \label{P3}
\end{equation}%
If the potential $P$ is defined by (\ref{P3}) then $\mathbf{y}\cdot \mathbf{%
\breve{V}}=\mathbf{y}\cdot \mathbf{V-y}\cdot \nabla P$ \ satisfies (\ref%
{ahat0}).
\end{proof}

Now we provide some explicit formulas for polynomial fields $\mathbf{V}$.

To obtain an explicit expression for $P$, we assume that $\mathbf{V}$ and \ $%
P$ \ are expanded into series of homogenious expressions:%
\begin{equation}
\mathbf{V}\left( \mathbf{y}\right) =\sum_{j}\mathbf{V}_{j}\left( \mathbf{y}%
\right) ,\quad P=\sum_{j}P_{j}\left( \mathbf{y}\right) ,\quad \mathbf{\breve{%
V}}\left( \mathbf{y}\right) =\sum_{j}\mathbf{\breve{V}}_{j}\left( \mathbf{y}%
\right) ,  \label{afihom}
\end{equation}%
where%
\begin{equation*}
\mathbf{V}_{j}\left( \zeta \mathbf{y}\right) =\zeta ^{j}\mathbf{A}_{j}\left( 
\mathbf{y}\right) ,\quad P_{j}\left( \zeta \mathbf{y}\right) =\zeta
^{j}\varphi _{j}\left( \mathbf{y}\right) .
\end{equation*}%
\ For a $j$-homogenious $\mathbf{V}_{j}$ we have 
\begin{equation*}
P_{j+1}\left( \mathbf{y}\right) =\int_{0}^{\left\vert \mathbf{y}\right\vert }%
\frac{1}{\left\vert \mathbf{y}\right\vert }\mathbf{y}\cdot \mathbf{V}%
_{j}\left( \frac{1}{\left\vert \mathbf{y}\right\vert }\mathbf{y}r\right) \,%
\mathrm{d}r=\int_{0}^{\left\vert \mathbf{y}\right\vert }\frac{1}{\left\vert 
\mathbf{y}\right\vert }\frac{1}{\left\vert \mathbf{y}\right\vert ^{j}}%
\mathbf{y}\cdot \mathbf{V}_{j}\left( \mathbf{y}\right) r^{j}\,\mathrm{d}r,
\end{equation*}%
implying%
\begin{equation}
P_{j+1}\left( \mathbf{y}\right) =\frac{1}{\left( j+1\right) \left\vert 
\mathbf{y}\right\vert ^{j+1}}\mathbf{y}\cdot \mathbf{V}_{j}\left( \mathbf{y}%
\right) \left\vert \mathbf{y}\right\vert ^{j+1}=\frac{1}{j+1}\mathbf{y}\cdot 
\mathbf{V}_{j}\left( \mathbf{y}\right) .  \label{Pj1}
\end{equation}%
In particular, the zero order term $\mathbf{V}_{0}$ corresponds to 
\begin{equation}
P_{1}\left( \mathbf{y}\right) =\mathbf{V}\left( 0\right) \cdot \mathbf{%
y,\quad \breve{V}}_{0}\left( \mathbf{y}\right) =0.  \label{ap1}
\end{equation}%
The first order one\ corresponds to 
\begin{equation}
P_{2}\left( \mathbf{y}\right) =\frac{1}{2}\mathbf{V}_{1}\left( \mathbf{y}%
\right) \cdot \mathbf{y}  \label{ap2}
\end{equation}%
and 
\begin{equation}
\nabla P_{2}\left( \mathbf{y}\right) =\frac{1}{2}\mathbf{V}_{1}\left( 
\mathbf{y}\right) +\frac{1}{2}\mathbf{V}_{1}^{\mathrm{T}}\left( \mathbf{y}%
\right) .  \label{ap2p}
\end{equation}%
\ Obviously, $\nabla P_{2}\left( \mathbf{y}\right) $ coincides with the
symmetric part of the linear transformation $\mathbf{V}_{1}\left( \mathbf{y}%
\right) $, and 
\begin{equation}
\mathbf{\breve{V}}_{1}\left( \mathbf{y}\right) =\frac{1}{2}\mathbf{V}%
_{1}\left( \mathbf{y}\right) -\frac{1}{2}\mathbf{V}_{1}^{\mathrm{T}}\left( 
\mathbf{y}\right)  \label{vhat}
\end{equation}%
coincides with the anti-symmetric part. For higher values of $j$ we have%
\begin{equation*}
\nabla P_{j+1}\left( \mathbf{y}\right) =\frac{1}{j+1}\nabla \left( \mathbf{y}%
\cdot \mathbf{V}_{j}\left( \mathbf{y}\right) \right) .
\end{equation*}%
Using vector calculus we obtain%
\begin{eqnarray*}
\nabla \left( \mathbf{y}\cdot \mathbf{V}_{j}\left( \mathbf{y}\right) \right)
&=&\left( \mathbf{y}\cdot \nabla \right) \mathbf{V}_{j}+\left( \mathbf{V}%
_{j}\cdot \nabla \right) \mathbf{y}+\mathbf{y}\times \left( \nabla \times 
\mathbf{V}_{j}\right) +\mathbf{V}_{j}\times \left( \nabla \times \mathbf{y}%
\right) \\
&=&\left( \mathbf{y}\cdot \nabla \right) \mathbf{V}_{j}+\mathbf{V}_{j}+%
\mathbf{y}\times \left( \nabla \times \mathbf{V}_{j}\right)
\end{eqnarray*}%
where, by Euler's identity for homogenious functions, 
\begin{equation*}
\left( \mathbf{y}\cdot \nabla \right) \mathbf{V}_{j}\left( \mathbf{y}\right)
=j\mathbf{V}_{j}\left( \mathbf{y}\right) .
\end{equation*}%
Hence%
\begin{gather*}
\nabla P_{j+1}\left( \mathbf{y}\right) =\frac{1}{j+1}\nabla \left( \mathbf{y}%
\cdot \mathbf{V}_{j}\left( \mathbf{y}\right) \right) =\frac{1}{j+1}\left(
\left( \mathbf{y}\cdot \nabla \right) \mathbf{V}_{j}+\mathbf{V}_{j}+\mathbf{y%
}\times \left( \nabla \times \mathbf{V}_{j}\right) \right) = \\
\frac{1}{j+1}\left( j\mathbf{V}_{j}+\mathbf{V}_{j}+\mathbf{y}\times \left(
\nabla \times \mathbf{V}_{j}\right) \right) =\mathbf{V}_{j}+\frac{1}{j+1}%
\mathbf{y}\times \left( \nabla \times \mathbf{V}_{j}\right) ,
\end{gather*}%
\begin{equation}
\mathbf{\breve{V}}_{j}\left( \mathbf{y}\right) =\mathbf{V}_{j}\left( \mathbf{%
y}\right) -\nabla P_{j+1}\left( \mathbf{y}\right) =-\frac{1}{j+1}\mathbf{y}%
\times \left( \nabla \times \mathbf{V}_{j}\right) .  \label{vhatj}
\end{equation}

\textbf{Acknowledgment.} The research was supported through Dr. A. Nachman
of the U.S. Air Force Office of Scientific Research (AFOSR), under grant
number FA9550-11-1-0163.


\begin{thebibliography}{99}
\bibitem{Ap} T. D'Aprile and D. Mugnai, Solitary waves for nonlinear
Klein--Gordon--Maxwell and Schr%
\"{}%
odinger--Maxwell equations, Proc. Roy. Soc. Edinburgh Sect. A 134(5),
893--906 (2004).

\bibitem{BronskiJer00} J. C. Bronski and R. L. Jerrard, \textsl{Soliton
dynamics in a potential}, Math. Res. Lett. 7, 329--342, (2000).

\bibitem{Appel Kiessling} \newblock W. Appel and M. Kiessling, \textsl{Mass
and Spin Renormalization in Lorentz Electrodynamics}, Ann. Phys., \textbf{289%
}, 24--83, (2001).

\bibitem{BF1}  \newblock A. Babin and A. Figotin, \newblock\textsl{Linear
superposition in nonlinear wave dynamics}, Reviews Math. Phys. 18 (2006),
no. 9, 971-1053.

\bibitem{BF2}  \newblock A. Babin and A.Figotin, \textsl{Wavepacket
Preservation under Nonlinear Evolution}, Commun. Math. Phys. 278, 329--384
(2008).

\bibitem{BF3} \newblock A. Babin and A. Figotin, \textsl{Nonlinear Dynamics
of a System of Particle-Like Wavepackets}, Instability in Models Connected
with Fluid Flows, Ed. C. Bardos and A. Fursikov, International Mathematical
Series, Vol. 6, Springer, 2008.

\bibitem{BF5} \newblock A. Babin and A. Figotin, \textsl{Wave-corpuscle
mechanics for electric charges}, J. Stat. Phys., \textbf{138}: 912--954,
(2010).

\bibitem{BF6} Babin A. and Figotin A., \textsl{Some mathematical problems in
a neoclassical theory of electric charges}, Discrete and Continuous
Dynamical Systems A, \textbf{27}(4), 1283-1326 (2010).

\bibitem{BF7} Babin A. and Figotin A., \textsl{Electrodynamics of balanced
charges}, Found. Phys., \textbf{41}: 242--260, (2011).

\bibitem{BF8} Babin A. and Figotin A., \textsl{Relativistic dynamics of
accelerating particles derived from field equations},\ Found. Phys., \textbf{%
42}: 996--1014, (2012).

\bibitem{BF9} Babin A. and Figotin A., \textsl{Relativistic Point Dynamics
and Einstein Formula as a Property of Localized Solutions of a Nonlinear
Klein-Gordon Equation}, Comm. Math. Phys. 322, 453-499, (2013).

\bibitem{BambusiG93} \newblock D. Bambusi, L. Galgani, \textsl{Some Rigorous
Results on the Pauli-Fierz Model of Classical Electrodynamics}, Ann. Inst.
H. Poincar\'{e}, Phys. Th\'{e}or. 58 155-171 (1993).

\bibitem{Barut}  \newblock A. Barut, \textsl{Electrodynamics and Classical
Theory of Fields and Particles}, Dover, 1980.

\bibitem{BerestyckiLions83I} \newblock H. Berestycki; P.-L. Lions, \textsl{%
Nonlinear scalar field equations. I. Existence of a ground state}, Arch.
Rational Mech. Anal. 82 (4) (1983) 313--345.

\bibitem{BerestyckiLions83II}  \newblock H. Berestycki; P.-L. Lions, \textsl{%
Nonlinear scalar field equations. II. Existence of infinitely many solutions,%
} Arch. Rational Mech. Anal. 82 (1983), no. 4, 347--375.

\bibitem{DelPinoDolbeault03}  \newblock M. Del Pino; J. Dolbeault, \textsl{%
The optimal Euclidean }$L_{p}$\textsl{\ -Sobolev logarithmic inequality}, J.
Funct. Anal. 197 (2003), no. 1, 151--161.

\bibitem{Bialynicki}  \newblock I. Bialynicki-Birula; J. Mycielski, \textsl{%
Nonlinear Wave Mechanics}, Annals of Physics, \textbf{100}, 62-93, (1976).

\bibitem{Bialynicki1} \newblock I. Bialynicki-Birula; J.Mycielski, \textsl{%
Gaussons: Solitons of the Logarithmic Schr\"{o}dinger Equation}, Physica
Scripta 20, 539-544 (1979).

\bibitem{Cazenave83}  \newblock T. Cazenave, \textsl{Stable solutions of the
logarithmic Schr\"{o}dinger equation}, Nonlinear Anal. 7 (1983), 1127-1140.

\bibitem{Cazenave03}  \newblock T. Cazenave, \textsl{Semilinear Schr\"{o}%
dinger equations}, Courant Lecture Notes in Mathematics, 10 AMS, Providence
RI, 2003.

\bibitem{CazenaveHaraux80}  \newblock T. Cazenave and A. Haraux, \textsl{%
\'{E}quations d'\'{e}volution avec non lin\'{e}arit\'{e} logarithmique},
Ann. Fac. Sci. Toulouse Math. (5) 2 (1980), no. 1, 21--51.

\bibitem{CazenaveLions82}  \newblock T. Cazenave; P.-L.Lions, \textsl{%
Orbital stability of standing waves for some nonlinear Schr\"{o}dinger
equations}, Comm. Math. Phys. 85 (1982), no. 4, 549--561.

\bibitem{Einstein05a} Einstein, Albert (1905). "Ist die Tr\"{a}gheit eines K%
\"{o}rpers von seinem Energieinhalt abh\"{a}ngig?". Annalen der Physik 18
(13): 639--641. English translation: "Does the Inertia of a Body Depend Upon
Its Energy Content?". Translation by George Barker Jeffery and Wilfrid
Perrett in The Principle of Relativity, London: Methuen and Company, Ltd.
(1923).

\bibitem{Feynman III} Feynman R., Leighton R. and Sands M., \textsl{The
Feynman Lectures on Physics}, Addison-Wesley, Reading, Vol. III, 1965.

\bibitem{FrTY} J. Frohlich, T.-P. Tsai, and H.-T. Yau, \textsl{On the
point-particle (Newtonian) limit of the non-linear Hartree equation}, Comm.
Math. Phys., 225, 223--274 (2002).

\bibitem{Goldstein} Goldstein H., Poole C., Safko J., \textsl{Classical
Mechanics}, 3rd ed., Addison-Wesley, 2000.

\bibitem{ImaikinKV06}  \newblock V. Imaikin, A. Komech, B. Vainberg, \textsl{%
On scattering of solitons for the Klein-Gordon equation coupled to a particle%
}, Comm. Math. Phys. 268 (2006), no. 2, 321-367. arXiv:math.AP/0609205

\bibitem{Itzykson Zuber} Itzykson C. and Zuber J., \textsl{Quantum Field
Theory}, McGraw-Hill, 1980.

\bibitem{JonssonFGS} B. Jonsson, J. Fr%
\"{}%
ohlich, S. Gustafson and I. M. Sigal, \textsl{Long time motion of NLS
solitary waves in a confining potential}, Ann. Henri Poincar%
\'{}%
e 7, 621--660 (2006).v

\bibitem{Heid}  \newblock M. Heid, H. Heinz and T. Weth, \textsl{Nonlinear
Eigenvalue Problems of Schr%
\"{}%
odinger Type Admitting Eigenfunctions with Given Spectral Characteristics},
Math. Nachr. \textbf{242} (2002), 91 -- 118.

\bibitem{Jackson}  \newblock J. Jackson, \textsl{Classical Electrodynamics},
3rd Edition, Wiley, 1999.

\bibitem{Kato89}  \newblock T. Kato, \textsl{Nonlinear Schr\"{o}dinger
equations}, in "Schr\"{o}dinger operators", (H. Holden and A Jensen, eds.),
Lecture Notes in Physics 345, Springer Verlag 1989.

\bibitem{Kiessling 1}  \newblock M. Kiessling, \textsl{Electromagnetic Field
Theory without Divergence Problems 1. The Born Legacy}, J. Stat. Physics, 
\textbf{116}: 1057-1122 (2004).

\bibitem{Kiessling2}  \newblock M. Kiessling, \textsl{Quantum Abraham models
with de Broglie-Bohm laws of quantum motion}, e-print available online at
arXiv:physics/0604069v2.

\bibitem{Komech05} \newblock A. Komech, \textsl{Lectures on Quantum
Mechanics (nonlinear PDE point of view)}, Lecture Notes of the Max Planck
Institute for Mathematics in the Sciences, LN 25/2005, Leipzig, 2005.
http://www.mis.mpg.de/preprints/ln/lecturenote-2505-abstr.html.

\bibitem{Komech09}  \newblock A.I. Komech, A.A. Komech, \textsl{Global
Attraction to Solitary Waves in Models Based on the Klein-Gordon Equation},
SIGMA, Symmetry Integrability Geom. Methods Appl. 4 (2008), Paper 010, 23
pages, electronic only. http://www.emis.de/journals/SIGMA/2008/

\bibitem{KKunzeSpohn}  \newblock A. Komech, M. Kunze, H. Spohn, \textsl{%
Effective Dynamics for a mechanical particle coupled to a wave field}, Comm.
Math. Phys. 203 (1999), 1-19.

\bibitem{Lanczos VPM} Lanczos C. \textsl{The Variational Principles of
Mechanics}, 4th ed., Dover, 1986.

\bibitem{LandauLif F} Landau L. and Lifshitz E., \textsl{The classical
theory of fields, Pergamon}, Oxford, 1975.

\bibitem{LongS08} E. Long and D. Stuart, \textsl{Effective dynamics for
solitons in the nonlinear Klein-Gordon-Maxwell system and the Lorentz force
law}, Rev. Math. Phys., 21, 459--510 (2009).

\bibitem{MaslovFedorjuk} Maslov V. P., Fedoriuk M. V., \textsl{%
Semi-Classical Approximation in Quantum Mechanics, }Reidel, Boston (1981).

\bibitem{Moller} M\NEG{o}ller C., \textsl{The Theory of Relativity}, 2nd
edition, Oxford, 1982.

\bibitem{Morse Feshbach I} Morse P. and Feshbach H.,\textsl{\ Methods of
Theoretical Physics}, Vol. I, McGraw-Hill, 1953.

\bibitem{Nayfeh} Nayfeh A., \textsl{Perturbation methods}, Wiley, 1973.

\bibitem{Pearle1} \newblock P. Pearle \textsl{Classical Electron Models}, in%
\textsl{\ Electromagnetism Paths to Research}, D. Teplitz, ed. (Plenum, New
York, 1982), pp. 211--295.

\bibitem{Poincare} \newblock H. Poincar\'{e}, Comptes Rend. 140, 1504
(1905); Rendiconti del Circolo Matematico di Palermo, \textbf{21}, 129--176,
1906.

\bibitem{Rohrlich}  \newblock F. Rohrlich, \textsl{Classical Charged
Particles}, Addison-Wesley, 3d ed., 2007.

\bibitem{Schwinger}  \newblock J. Schwinger,\textsl{\ Electromagnetic Mass
Revisited}, Foundations of Physics, Vol. 13, No. 3, pp. 373-383, 1983.

\bibitem{Spohn}  \newblock H. Spohn, \textsl{Dynamics of Charged Particles
and Their Radiation Field}, Cambridge Univ. Press, 2004.

\bibitem{Stachel EBZ} Stachel J., \textsl{Einstein from B to Z}, Burkhouser,
2002.

\bibitem{Sulem}  \newblock C. Sulem and P. Sulem, \textsl{The nonlinear Schr%
\"{o}dinger equation. Self-focusing and wave collapse}, Springer, 1999.
\end{thebibliography}
\end{document}